\documentclass[DIV=13,abstract=true,paper=a4,fontsize=11pt,parskip=half]{scrartcl}

\RequirePackage{booktabs}
\RequirePackage{dsfont}
\RequirePackage{microtype}
\usepackage{lmodern}
\usepackage{inputenc}
\usepackage[T1]{fontenc}
\usepackage[USenglish]{babel}
\usepackage{latexsym}
\usepackage{amsmath,amsfonts,amssymb,amsthm,amsopn}
\usepackage[centercolon]{mathtools}
\usepackage{fixmath}
\usepackage{braket}
\usepackage{graphicx}
\usepackage{epstopdf}
\usepackage[svgnames]{xcolor}
\usepackage{tikz}
\usepackage{pgfplots,pgfplotstable}

\pgfplotsset{compat=1.7}
\usepgfplotslibrary{statistics}
\usetikzlibrary{pgfplots.statistics}
\usepackage{colortbl}
\usepackage{multirow}
\usepackage{array}
\usepackage{makecell}
\usepackage{enumerate}
\usepackage[pdftex,pagebackref=true,bookmarks=true,pdfpagelabels=true,plainpages=false,
hypertexnames=false,bookmarksnumbered=true,bookmarksopen=true,
pdfstartview=Fit,colorlinks,
linkcolor=DarkGreen,citecolor=DarkBlue,urlcolor=DarkRed]{hyperref}
\usepackage[algo2e,ruled,linesnumbered]{algorithm2e}
\SetKw{KwTerminate}{terminate}

\SetCommentSty{mycommfont}
\usepackage[nameinlink,capitalise]{cleveref}
\crefformat{equation}{(#2#1#3)}
\crefname{assumption}{Assumption}{Assumptions}
\Crefname{assumption}{Assumption}{Assumptions}
\crefname{figure}{Figure}{Figures}
\theoremstyle{plain}
\newtheorem{theorem}{Theorem}[section]
\newtheorem{lemma}[theorem]{Lemma}
\newtheorem{corollary}[theorem]{Corollary}
\theoremstyle{definition}
\newtheorem{definition}[theorem]{Definition}
\newtheorem{assumption}[theorem]{Assumption}
\theoremstyle{remark}
\newtheorem{remark}[theorem]{Remark}

\providecommand{\MSC}[1]{\textbf{Mathematics Subject Classification} #1}

\newdimen\fwd

\DeclareMathOperator{\proj}{\Pi}
\DeclareMathOperator*{\upd}{Update}

\newcommand{\norm}[1]{\ensuremath{\lVert #1\rVert}}
\newcommand{\Norm}[1]{\ensuremath{\left\lVert #1\right\rVert}}
\newcommand{\xopt}{\ensuremath{x^\ast}}
\newcommand{\Lo}{{\L}ojasiewicz}
\newcommand{\KL}{Kurdyka--\Lo}

\newcommand{\PetraM}{\textsc{S-LBFGS-M}}
\newcommand{\PetraP}{\textsc{S-LBFGS-P}}
\newcommand{\Ballop}[1]{\ensuremath{\mathbb{B}_{#1}}}

\newcommand{\Lin}{\ensuremath{\mathcal{L}}}
\newcommand{\LinPDX}{\ensuremath{\mathcal{L}_+(\CX)}}
\newcommand{\tl}{\ensuremath{\epsilon}}
\newcommand{\brk}{\ensuremath{break}}
\newcommand{\outp}{\ensuremath{output}}

\newcommand{\CD}{{\cal D}}
\newcommand{\CS}{{\cal S}}
\newcommand{\CJ}{{\cal J}}
\newcommand{\CJopt}{{\cal J^\ast}}
\newcommand{\CX}{{\cal X}}
\newcommand{\CN}{{\cal N}}
\newcommand{\R}{\mathbb{R}} 
\newcommand{\N}{\mathbb{N}}

\newcommand{\Hy}{\ensuremath{y}}
\newcommand{\Hs}{\ensuremath{s}}
\newcommand{\Hp}{\ensuremath{p}}
\newcommand{\Bp}{\ensuremath{s}}
\newcommand{\Bz}{\ensuremath{z}}
\newcommand{\Bu}{\ensuremath{u}}
\newcommand{\GM}{\ensuremath{g}}
\newcommand{\Adap}{{\mathrm{Adap}}}

\newcommand{\FAIRTEXT}{{\textsc{FAIR}}}
\newcommand{\LBFGS}{\mbox{L-BFGS}}
\newcommand{\eps}{\varepsilon}

\newcommand{\wbp}{w^s}
\newcommand{\wgm}{w^g}
\newcommand{\wbz}{w^z}

\title{A structured L-BFGS method and its application to inverse problems}

\author{Florian Mannel\thanks{Institute of Mathematics and Image Computing, University of Lübeck, 
		Lübeck, Germany (\href{mailto:florian.mannel@uni-luebeck.de}{florian.mannel@uni-luebeck.de}/\href{mailto:hariom85@gmail.com}{hariom85@gmail.com})}
	\,\href{https://orcid.org/0000-0001-9042-0428}{\includegraphics[height=.35cm]{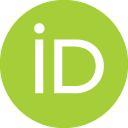}}
	\and Hari Om Aggrawal\footnotemark[1] \,\href{https://orcid.org/0000-0002-8892-5993}{\includegraphics[height=.35cm]{ORCIDiD128x128.png}} 
	\and Jan Modersitzki\thanks{Fraunhofer Institute for Digital Medicine MEVIS, Lübeck, Germany (\href{mailto:jan.modersitzki@mevis.fraunhofer.de}{jan.modersitzki@mevis.fraunhofer.de})}}

\date{Preprint, \today}

\ifpdf
\hypersetup{
	pdftitle={A structured L-BFGS method and its application to inverse problems},
	pdfsubject={Accelerating L-BFGS by using readily accessible Hessian information},
	pdfauthor={F. Mannel}
}
\fi

\begin{document}

\maketitle

\begin{abstract}
	Many inverse problems are phrased as optimization problems in which the objective function is the sum of a data-fidelity term and a regularization.
	Often, the Hessian of the fidelity term is computationally unavailable while the Hessian of the regularizer allows for cheap matrix-vector products. 
	In this paper, we study an \LBFGS~method that takes advantage of this structure. 
	We show that the method converges globally without convexity assumptions and that the convergence is linear under a \KL-type inequality. 
	In addition, we prove linear convergence to cluster points near which the objective function is strongly convex. 
	To the best of our knowledge, this is the first time that linear convergence of an \LBFGS~method is established in a non-convex setting. 
	The convergence analysis is carried out in infinite dimensional Hilbert space, which is appropriate for inverse problems but has not been done before.
	Numerical results show that the new method outperforms other structured \LBFGS~methods and classical \LBFGS~on non-convex real-life problems from medical image registration. It also compares favorably with classical \LBFGS~on ill-conditioned quadratic model problems. 
	An implementation of the method is freely available. 
\end{abstract}


\MSC{65J22 $\cdotp$ 65K05 $\cdotp$ 65K10 $\cdotp$ 90C06 $\cdotp$ 90C26 $\cdotp$ 90C30 $\cdotp$ 90C48 $\cdotp$ 90C53 $\cdotp$ 90C90}



\section{Introduction}
	
	\subsection{Topic and main contributions of the paper}
	
	Many inverse problems are phrased as unconstrained optimization problems 
	\begin{equation*}
		\min_{x\in\CX}\,\CJ(x)
	\end{equation*}
	with a cost function of the form
	\begin{equation}\label{eq_SO}
		\CJ:\CX\to\R,\qquad \CJ(x)=\CD(x)+\CS(x).
	\end{equation}
	Here, $\CD:\CX\rightarrow\R$ represents a data-fitting term, $\CS:\CX\rightarrow\R$ a regularizer, and $\CX$ a Hilbert space. 
	Typically, after discretization the data-fitting term has an ill-conditioned and dense Hessian for which even matrix-vector multiplications are computationally expensive, while the Hessian of the regularizer is positive definite, well-conditioned and sparse with computationally cheap matrix-vector products.
	The \LBFGS~method \cite{N80,Liu1989,BNS94} is one of the most widely used algorithms for large-scale inverse problems, but in its plain form 
	does not take into account possible splittings of the cost function such as \cref{eq_SO}. 
	
	In this paper we present an \LBFGS~method for objectives of the form \cref{eq_SO}
	that exploits the different structural properties of the terms $\CD$ and~$\CS$.
	We provide a very general convergence analysis for this novel method and we demonstrate on 
	real-world problems from medical image registration that the method is extremely effective in practice.
	More specifically, we prove the global convergence of the method for convex and for non-convex $\CJ$ and we show its linear convergence under a \KL-type inequality and also under a \emph{local} convexity assumption; results of this quality are new for \LBFGS-type methods.
	In the numerical experiments we find that the method significantly outperforms its competitors from \cite{BDLP22} and also its unstructured counterpart.
	
	The main motivation for this paper arose from the need to improve the performance of \LBFGS~on
	image registration problems \cite{FAIR09}. These problems are large-scale, highly non-convex inverse problems in which the structure \cref{eq_SO} is ubiquitous.
	In this context, the results of this paper cover regularizers such as $L_2$-norm based Tikhonov regularizers \cite{FAIR09,Hansen2010}, smooth total-variation norm \cite{Vogel2002}, quadratic forms of derivative based regularization, where $\CS(x)=\|Bx\|^2_{L_2}$ with a linear differential operator $B$, and the non-quadratic and highly non-convex hyperelastic regularizer~\cite{BMR13}. However, we stress that the convergence analysis of this paper is very general and not restricted to inverse problems.
	
	\subsection{Structured seed matrices for structured problems}\label{sec_introtwoloop}
	
	In this section we detail how we incorporate the structure of the objective into \LBFGS~and we discuss some related issues. 
	To this end, let us briefly summarize the construction of the \LBFGS~operator for $\min_{x\in\CX}\CJ(x)$; for details see \cite[Section~7.2]{Nocedal2006}.
	Given the current iterate $x_k$ and stored update vectors $\{(s_{j},y_{j})\}_{j=k-\ell}^{k-1}$ with $y_j^T s_j>0$ for all $j$, 
	the \LBFGS~operator $B_k\approx\nabla^2\CJ(x_{k})$ is obtained as $B_k:=B_k^{(\ell)}$ from the recursion
	\begin{equation}\label{eq_defLBFGSupd1}
		B_k^{(j+1)}:=B_k^{(j)} + \upd\bigl(s_{m+j},y_{m+j},B_k^{(j)}\bigr), \qquad j=0,\ldots,\ell-1.
	\end{equation}
	Here, $m:=k-\ell$, and for $(s,y)\in\CX\times\CX$ with $y^T s>0$ and positive definite $B$ the update is 
	\begin{equation}\label{eq_defU}
		\upd\bigl(s,y,B\bigr):=\frac{y y^T}{y^T s} - \frac{B s s^T B^T}{s^T B s}.
	\end{equation}
	A key observation is that the \emph{seed matrix $B_k^{(0)}$} needs to be specified in \cref{eq_defLBFGSupd1}. 
	The most widely used choice is $B_k^{(0)}=\tau_k I$ with a positive $\tau_k>0$, but it is well known 
	that this choice can lead to slow convergence on ill-conditioned problems. 
	It is natural to expect that if we make $B_k^{(0)}$ a better approximation of $\nabla^2\CJ(x_k)$, 
	the convergence will improve, particularly on ill-conditioned problems.
	To incorporate the structure given in \cref{eq_SO} we therefore propose to use a \emph{structured seed matrix} of the form
	\begin{equation}\label{eq_choiceseedmatrix}
		B_k^{(0)}=\tau_kI+S_k 
	\end{equation}
	where $\tau_k\geq 0$ and $S_k$ are selected such that $B_k^{(0)}$ is positive definite. 
	The main choice that we have in mind is $S_k=\nabla^2\CS(x_k)$, but other choices are also of interest,
	for instance if only approximations of $\nabla^2\CS(x_k)$ are available or if $\nabla^2\CS(x_k)$ is not positive semi-definite.
	
	\paragraph{Costs and benefits of a structured seed matrix}
	The search direction $d_k$ in \LBFGS~is obtained from $B_k d_k = -\nabla\CJ(x_k)$. It is a big advantage of the choice $B_k^{(0)}=\tau_k I$ 
	that $d_k$ can be computed without solving a linear system; this is due to the Sherman--Morrison formula and the fact that $B_k^{(0)}=\tau_k I$ is trivial to invert. 
	In contrast, using~\cref{eq_choiceseedmatrix} 
	requires to solve a linear system with $B_k^{(0)}=\tau_kI+S_k $ as coefficient matrix. 
	It is therefore critical for the performance of our method that such linear systems can be solved efficiently, be it iteratively or otherwise.
	For many regularizers in inverse problems the properties of $\nabla^2\CS(x_k)$
	imply that this assumption is valid for $S_k=\nabla^2\CS(x_k)$ or an approximation $S_k\approx\nabla^2\CS(x_k)$.
	Still, in view of the additional linear algebra costs, 
	the question arises whether using \cref{eq_choiceseedmatrix} can actually lower the \emph{run-time} of \LBFGS.
	While the answer to this question is problem-dependent in general, 
	extensive numerical evidence clearly shows that using a better approximation of the Hessian $\nabla^2\CJ(x_k)$ as seed matrix can be vastly superior to the standard choice $B_k^{(0)}=\tau_k I$ for large-scale real-world problems, cf. \cite{JBES04,Heldmann2006,FAIR09,KR13,ACG18,YGJ18,AM21,BDLP22}. In this paper we observe the same effect in the structured setting \cref{eq_SO} for choices of the form \cref{eq_choiceseedmatrix}. 
	
	\paragraph{The choice of the scaling factor in the structured \LBFGS~method}
	
	In classical \LBFGS, the scalar $\tau_k$ for $B_k^{(0)}=\tau_k I$ is often computed according to the Oren-Luenberger scaling strategy~\cite{Oren1982}. 
	The common choice $\tau_k=y_{k-1}^T y_{k-1}/ y_{k-1}^T s_{k-1}$ \cite[(7.20)]{Nocedal2006} is obtained in this way, for instance.
	We will apply the same strategy to \cref{eq_choiceseedmatrix} to derive suitable values for $\tau_k$. 
	These values, along with a novel adaptive choice, are included in the convergence analysis and in the numerical experiments.

	\subsection{Main convergence results for the structured \LBFGS~method}\label{sec_introconvana}
	
	Next we briefly discuss the main convergence results of this work. 
	Among others, this allows us to highlight that \emph{cautious updates} are an important ingredient of the new method. 
	
	In the available literature for \LBFGS-type and BFGS-type methods, global and linear convergence 
	is only established under the assumption of strong convexity 
	of the objective in the level set associated to the starting point; 
	cf., for instance, \cite{Liu1989,BN89,YSWHL18,GG19,VL19,SXBN22}. 
	This excludes many interesting objective functions and, in particular, is not appropriate for the highly non-convex nature of 
	many inverse problems, including those arising in medical image registration. 
	In this paper, we prove global and linear convergence in different settings, none of which assumes strong convexity of $\CJ$ in a level set. 
	For instance, we show $\lim_{k\to\infty}\norm{\nabla\CJ(x_k)}=0$ and that if
	$(x_k)$ has a cluster point $\xopt$ near which $\CJ$ is strongly convex, then 
	the entire sequence $(x_k)$ converges linearly to~$\xopt$; cf. \cref{thm_globconv} and \cref{thm_rateofconvstrongminimizer}. 
	This type of result, which is available for many other optimization methods but novel for \LBFGS, 
	is very valuable since it replaces the assumption of \emph{global} strong convexity by \emph{local} strong convexity. 
	The key to obtaining such a comparably strong result is to use \emph{cautious updates} \cite{LiFukucautiousupdates} to control the condition number of the Hessian approximations $(B_k)$. 
	While we defer the specification of the cautious updates that we use to \cref{sec_SLBFGS}, 
	we stress that they are an important contribution of the paper.
	To underline this point, note that in the existing literature for \LBFGS- and BFGS-type methods, global convergence for non-convex objectives 
	is often only established in the sense that $\liminf_{k\to\infty}\norm{\nabla\CJ(x_k)}=0$, e.g. in \cite{LF01,XLW13,KD15,Zh20,YWS20,YZZ22},
	whereas we prove $\lim_{k\to\infty}\norm{\nabla\CJ(x_k)}=0$. 
	We are not aware of works on \LBFGS~methods with cautious updates that
	prove this stronger form of global convergence. It is, however, proved in \cite{WMGL17} for a stochastic \LBFGS-type method, 
	in \cite{BJMT22} for sampled \LBFGS, and in \cite{KS23} for iteratively regularized \LBFGS,
	but linear convergence is not established in these works. 
	
	In \cref{thm_rateofconvPL} we prove that our structured \LBFGS~method 
	converges linearly under a \KL-type inequality. 
	This complements the aforementioned result on linear convergence because the \KL~condition, in contrast to local strong convexity, allows for 
	locally non-unique minimizers. To the best of our knowledge, this is the first time that linear convergence of an \LBFGS-type method is shown using a \KL~inequality. 
	
	Let us mention three more aspects of the theoretical results in this paper. 
	First, we allow $\ell=0$ in the \LBFGS~method, where $\ell$ is the number of stored update vectors, cf. \cref{eq_defLBFGSupd1}. This means that no updates are applied, hence $B_k=B_k^{(0)}=\tau_k I + S_k$. For the choice $S_k=0$ this is a \emph{Barzilai--Borwein method} \cite{BB88,DF05}.
	This class of methods has recently received renewed interest, e.g. 
	\cite{BDH19,AK22,MMPS22,CPRZ23}, but to the best of our knowledge this is the first time that a \emph{structured} Barzilai--Borwein method is considered. 
	Second, we carry out the convergence analysis in Hilbert spaces. This is appropriate for inverse problems, but
	has not been done before for \LBFGS, despite its frequent use for these problems.
	We acknowledge, however, that BFGS has been considered in infinite dimensional settings, e.g. in
	\cite{GriewankInfDimBFGS,Gilbert1989,MayorgaQuintanaInfDimQN,GruverSachsInfDimQN,Kupfer,PetraInfDimQN}.
	Third, the convergence analysis lowers the usual differentiability requirements in that it does not require $\CJ$ to be twice continuously differentiable. 
		
	\subsection{Related work}\label{lit_relwork}
	
	The idea to exploit the problem structure for constructing a better approximation of $\nabla^2\CJ(x_k)$ and use it as a seed matrix in \LBFGS~appears in various numerical studies \cite{JBES04,Heldmann2006,FAIR09,KR13,ACG18,YGJ18,AM21,BDLP22},
	sometimes under the name \emph{preconditioned \LBFGS}. 
	Together, these works cover a wide range of real-life problems like molecular energy minimization \cite{JBES04,KR13}, independent component analysis 
	\cite{ACG18}, medical image registration \cite{Heldmann2006,FAIR09,AM21}, logistic regression and optimal control \cite{BDLP22}.
	The authors consistently report significant speed ups on these large-scale problems over all methods that are used for comparison,
	including classical \LBFGS, Gauss--Newton and truncated Newton.
	However, in contrast to this paper they do not provide a convergence analysis and they do not use cautious updates in their methods.
	Also, except for \cite{AM21,BDLP22} the choice of $\tau_k$ is not studied in the numerical experiments. 
	For $B_k^{(0)}=\tau_k I$ this choice can meaningfully influence the numerical performance of \LBFGS, cf. for instance 
	the classical works \cite{Oren1982,Liu1989,Gilbert1989}, 
	so we investigate it more thoroughly than \cite{AM21,BDLP22}. 
	In comparison to our conference paper \cite{AM21} the method that we present here is also significantly more effective and we test it on a substantially larger problem set. 
	The structured methods from \cite{BDLP22}, furthermore, rely on a different formula for the quasi-Newton update and they use compact representations instead of the two-loop recursion that our method is built on; we discuss these aspects in detail in \cref{sec_difftoPetra}. 
	On image registration problems our method outperforms the methods from \cite{BDLP22}, cf. \cref{sec:experiments}.
	
	While we are not aware of further works on structured quasi-Newton methods in the \emph{limited memory} setting, structured variants of \emph{full memory} quasi-Newton methods have been studied, e.g., in general frameworks \cite{DW81,DW85,EM91,YY96}, for the BFGS method \cite{DMT89}, for the Powell--Symmetric--Broyden method \cite{La00}, and for Broyden's method \cite[Section~3]{HK92}, the latter also in a nonsmooth setting \cite{MR21,MR22}. 
	Most often, structured quasi-Newton methods have been considered for least squares problems in the form of Gauss-Newton-type schemes, 
	for instance in \cite{DGW81,DMT89,H94,H05,ZC10,WLQ10,MM19}. Among these, only \cite{H05} considers objectives of the form \cref{eq_SO}, but provides no convergence analysis and does not develop a structured \LBFGS~method. 
	
	\emph{Globalization} of \LBFGS-type methods is usually achieved by line search strategies, but trust region approaches and iterative regularization have also been studied, 
	see for instance \cite{BGZY17} and the references therein, respectively, \cite{L14,TP15,TSY22,KS23}. However, a rate of convergence has not been established for these methods.
		
	\subsection{Code availability}

	An implementation of our structured \LBFGS method that includes an example from the numerical section of this paper
	is freely available at \href{https://github.com/hariagr/SLBFGS}{https://github.com/hariagr/SLBFGS}.
	
	\subsection{Organization and notation}
	
	The paper is organized as follows. 
	To set the stage, \cref{sec:lbfgsIntro} recalls standard \LBFGS, followed 
	by the presentation of the structured \LBFGS~method in \cref{sec_SLBFGS}. 
	\Cref{sec_convana} provides the convergence analysis
	and \cref{sec:experiments} the numerical experiments. A summary is given in \cref{sec:conclusion}.
	
	We use $\N=\{1,2,3,\ldots\}$ and $\N_0=\N\cup\{0\}$.
	The scalar product of $x,y\in\CX$ is indicated by $x^T y$, and for $x\in\CX$ the linear functional
	$y\mapsto x^T y$ is denoted by $x^T$. 
	The induced norm is $\norm{x}$. 
	We write $A\in\Lin(\CX)$ to express the fact that $A:\CX\rightarrow\CX$ is a bounded linear mapping.
	

\enlargethispage{2\baselineskip}

\section{The classical \LBFGS~method}\label{sec:lbfgsIntro}
		
	In this section we briefly discuss \LBFGS, the most popular limited memory quasi-Newton method.
	It applies to the minimization of a 
	continuously differentiable function $\CJ:\CX\rightarrow\R$ on a Hilbert space $\CX$, 
	$\min_{x\in\CX} \CJ(x)$, with $\CJ$ not necessarily of the form \cref{eq_SO}. 
	We state the algorithm in terms of $H_k:=B_k^{-1}$ instead of $B_k$, 
	where $B_k$ is computed according to \cref{eq_defLBFGSupd1}; cf. Algorithm~\ref{algo:lbfgs} below.
	
	\begin{algorithm2e}[h!]
		\SetAlgoRefName{LBFGS}
		\DontPrintSemicolon
		\caption{Inverse \LBFGS~method, cf. \cite[Alg.~7.5]{Nocedal2006}}
		\label{algo:lbfgs}
		\KwIn{$x_0\in \CX$, $\ell>0$, $\tl\in[0,\infty)$}
		\For{$k=0,1,2,\ldots$}
		{
			Let $m:=\max\{0,k-\ell\}$\\
			\lIf{$\norm{\nabla\CJ(x_k)}\leq\tl$}{\outp~$x_k$ and \brk}
			Choose symmetric positive definite $H_k^{(0)}\in\Lin(\CX)$\tcp*{Often, $H_k^{(0)}=\hat\tau_k I$ with $\hat\tau_k>0$}
			Compute $d_k:=-H_k\nabla\CJ(x_k)$ from $H_k^{(0)}$ and the stored pairs $\{(s_j,y_j)\}_{j=m}^{k-1}$ using the two-loop recursion \cite[Algorithm~7.4]{Nocedal2006}\label{line_tlrunstrbfgs}\\ 
			Compute step length $\alpha_k>0$ using a line search\\
			Let $s_k:=\alpha_k d_k$, $x_{k+1}:=x_k + s_k$, $y_k:=\nabla\CJ(x_{k+1})-\nabla\CJ(x_k)$\\
			\lIf{$y_k^T s_k > 0$}{append $(s_k,y_k)$ to storage}
			\lIf{$k\geq \ell$}{remove $(s_{m},y_{m})$ from storage}
		}
	\end{algorithm2e}	
	
	\subsection{Step size selection}\label{sec_stepsizes}
	
	The step length $\alpha_k$ in Algorithm~\ref{algo:lbfgs} is usually computed in such a way that it satisfies the weak or the strong Wolfe--Powell conditions. 
	However, some authors determine $\alpha_k$ by backtracking until the Armijo condition holds. For later reference let us make this explicit. 
	For Armijo with backtracking we fix constants $\beta,\sigma\in(0,1)$. The step size $\alpha_k>0$ for the iterate $x_k$ with associated descent direction~$d_k$ is the largest number in $\{1,\beta,\beta^2,\ldots\}$ such that $x_{k+1}=x_k+\alpha_k d_k$ satisfies
	\begin{equation}\label{eq_armijocond}
	\CJ(x_{k+1}) \leq \CJ(x_k) + \alpha_k\sigma\nabla\CJ(x_k)^T d_k.
	\end{equation}
	For the Wolfe--Powell conditions, respectively, the strong Wolfe--Powell conditions we additionally fix $\eta\in(\sigma,1)$. 
	We accept any step size $\alpha_k>0$ that satisfies \cref{eq_armijocond} and
	\begin{equation}\label{eq_WolfePowellcond}
	\nabla\CJ(x_{k+1})^T d_k \geq \eta\nabla\CJ(x_k)^T d_k, \quad\text{ respectively, }\quad
	\left\lvert\nabla\CJ(x_{k+1})^T d_k\right\rvert \leq \eta\left\lvert\nabla\CJ(x_k)^T d_k\right\rvert.
	\end{equation}
	
	\subsection{Choice of the scaling factor \texorpdfstring{$\hat\tau_k$}{hat tau_k}}
	
	The two most commonly used scaling factors 
	for $H_k^{(0)}=\hat\tau_k I$ in inverse \LBFGS~methods are the BB scaling factors introduced by Barzilai and Borwein~\cite{BB88}.
	
	\begin{definition}[BB scaling]\label{def:tau:H}
	For $s_k,y_k\in\CX$ with~$\rho_k:=y_k^T s_k\neq 0$ we define
	\begin{equation*} 
	\hat\tau_{k+1}^\Hy:=\rho_k/\norm{y_k}^2 \qquad\text{ and }\qquad
	\hat\tau_{k+1}^\Hs:=\norm{s_k}^2/\rho_k.
	\end{equation*}
	\end{definition}
	
	A comprehensive treatment of the two scaling factors in a finite dimensional setting including numerical experiments is contained in \cite[Section~4.1]{Gilbert1989}. Its conclusion is that $\hat\tau_k^\Hy$ is preferable for inverse \LBFGS,
	and this view has since prevailed, cf. e.g. \cite[(7.20)]{Nocedal2006}.
		
	It can be proven in the same way as in finite dimensions \cite[Section~4.1]{Gilbert1989} that the BB factors are related to each other and to the inverse of the \emph{average Hessian} $\overline{\nabla^2\CJ_k}$ as follows.

	\begin{lemma}\label{lem:tau}
	Let $\CJ:\CX\rightarrow\R$ be twice continuously differentiable. Let $x_k,x_{k+1}\in\CX$ be such that 
	$y_k:=\nabla\CJ(x_{k+1})-\nabla\CJ(x_k)$ and $s_k:=x_{k+1}-x_k$ satisfy
	$y_k^T s_k>0$. Let 
	\begin{equation*}
	\overline{\nabla^2\CJ_k}:=\int_0^1 \nabla^2\CJ(x_k+t s_k)\,\mathrm{d}t 
	\end{equation*}
	be positive semi-definite. 
	Then for $\hat \tau_{k+1}^\Hy$ and $\hat\tau_{k+1}^\Hs$ from \cref{def:tau:H} we have 
	\begin{equation*}
	\Norm{\overline{\nabla^2\CJ_k}}^{-1} 
	\le\hat\tau_{k+1}^\Hy\le\hat\tau_{k+1}^\Hs
	\le \Norm{\Bigl(\overline{\nabla^2\CJ_k}\Bigr)^{-1}},
	\end{equation*}
	where the upper bound requires that $\overline{\nabla^2\CJ_k}$ is invertible.
	\end{lemma}
	
	\begin{remark}\label{rem_scalingfactorssecantequation}
	Since $\norm{\overline{\nabla^2\CJ_k}}^{-1}$ and $\norm{(\overline{\nabla^2\CJ_k})^{-1}}$
	correspond to the smallest, respectively, largest eigenvalue of $(\overline{\nabla^2\CJ_k})^{-1}$,
	\cref{lem:tau} shows that $\hat\tau_{k+1}^\Hy$ and $\hat\tau_{k+1}^\Hs$ approximate the spectrum of the inverse of the averaged Hessian. 
	Furthermore, it is easy to see that $\hat\tau_{k+1}^\Hy$ minimizes $\hat\tau\mapsto\norm{\hat\tau y_k-s_k}$ for $\hat\tau\in\R$.
	That is, $\hat\tau_{k+1}^\Hy$ yields the least squares fit of $y_k$ to $s_k$ that approximately realizes
	the secant equation $H_{k+1}^{(0)}y_k=s_k$ for $H_{k+1}^{(0)}=\hat\tau_{k+1} I$. 
	The secant equation plays a crucial role
	in designing quasi-Newton methods, cf., e.g., \cite{DS79,Nocedal2006},
	and we will use the same principle for our \LBFGS~method in \cref{sec_choiceoftaukinLBFGSM}.
	Similar comments apply to $\hat\tau_{k+1}^\Hs$, which minimizes $\norm{y_k-s_k/\hat\tau}$ 
	and may be viewed as a least squares fit of $s_k$ to $y_k$, cf. also \cref{fig:pz}.
	\end{remark}
	

\enlargethispage{2\baselineskip}

\section{The structured \LBFGS~method}\label{sec_SLBFGS}
	
	The proposed structured \LBFGS~algorithm for $\min_{x\in\CX}\CJ(x)$ with $\CJ$ having the form \cref{eq_SO} and the seed matrix chosen according to 
	\cref{eq_choiceseedmatrix} is summarized as Algorithm~\ref{alg_hybrid}.
	The main differences to Algorithm~\ref{algo:lbfgs} are the choice of the seed matrix
	and the two cautious updates that affect if $(s_k,y_k)$ enters the storage and that restrict the choice of $\tau_{k+1}$.  
	In the following subsections we provide the missing specifications for Algorithm~\ref{alg_hybrid} and we comment on it. 
	
	\begin{algorithm2e}
		\SetAlgoRefName{SLBFGS}
		\DontPrintSemicolon
		\caption{Structured inverse \LBFGS~method}
		\label{alg_hybrid}
		\KwIn{$x_0\in \CX$, $\tl\geq 0$, $\ell\in\N_0$, $c_0\geq 0$, $C_0\in [c_0,\infty]$, $c_s,c_1,c_2>0$}
		Let $\tau_0\in(0,\infty)$\\
		\For{$k=0,1,2,\ldots$}
		{
			Let $m:=\max\{0,k-\ell\}$\\
			Choose symmetric positive semi-definite $S_k\in\Lin(\CX)$\\ 
			Let $B_k^{(0)}:=\tau_kI+S_k$\tcp*{choice of seed matrix}
			Compute $d_k:=-H_k\nabla\CJ(x_k)$ from $B_k^{(0)}$ and the stored pairs $\{(s_j,y_j)\}_{j=m}^{k-1}$ using the two-loop recursion \cite[Algorithm~7.4]{Nocedal2006}\label{line_tlr}\\
			Compute step length $\alpha_k>0$ using a line search\label{line_nols}\\ 
			Let $s_k:=\alpha_k d_k$, $x_{k+1}:=x_k + s_k$, $y_k:=\nabla\CJ(x_{k+1})-\nabla\CJ(x_k)$\\
			\lIf{$y_k^T s_k > c_s\norm{s_k}^2$}{append $(s_k,y_k)$ to storage\label{line_acceptanceofysforstorage}\tcp*[f]{cautious update~1}}
			\lIf{$k\geq\ell$}{remove $(s_m,y_m)$ from storage}
			\lIf{$\norm{\nabla\CJ(x_{k+1})}\leq\tl$}{\outp~$x_{k+1}$ and \brk\label{line_term}}
			Let $z_k:=y_k-S_{k+1}s_k$\label{line_zk}\\
			Let $\omega_{k+1}^l:=\min\{c_0,c_1\norm{\nabla\CJ(x_{k+1})}^{c_2}\}$ and $\omega_{k+1}^u:=\max\{C_0,(c_1\norm{\nabla\CJ(x_{k+1})}^{c_2})^{-1}\}$\label{line_defomegakplusone}\\
			\lIf{$z_k^Ts_k>0$}{choose $\tau_{k+1}\in [\tau^\Bp,\tau^\Bz]$, where $\tau^\Bp$ and $\tau^\Bz$ are computed according to \cref{def:tau:B} using
				$(s,z)=(s_k,z_k)$,
				$\tau_{\min}=\omega_{k+1}^l$, $\tau_{\max}=\omega_{k+1}^u$\label{line_choicetauk1}\tcp*[f]{cautious update~2}}
			\lElse{choose $\tau_{k+1}\in [\tau^\Bp,\tau^\GM]$, where $\tau^\Bp$ and $\tau^\GM$ are computed according to \cref{def:tau:B} using
				$(s,z)=(s_k,z_k)$, $\tau_{\min}=\omega_{k+1}^l$, $\tau_{\max}=\omega_{k+1}^u$\label{line_choicetauk2}\tcp*[f]{cautious update~2}}
		}
	\end{algorithm2e}

	\subsection{Choice of the scaling factor \texorpdfstring{$\tau_k$}{tau_k}}\label{sec_choiceoftaukinLBFGSM}
	
	A key element of our method is to use 
	\begin{equation}\label{eq_choiceofseedmatrixgeneralversion}
		B_k^{(0)} = \tau_k I + S_k
	\end{equation}
	as seed matrix, where $\tau_k\geq 0$. Since it is well-known that the choice of the corresponding parameter $\hat\tau_k$ in Algorithm~\ref{algo:lbfgs}
	can have a material influence on the performance, we devote this section to the derivation of suitable choices for $\tau_k$ in Algorithm~\ref{alg_hybrid}.
	
	Following the Oren-Luenberger scaling strategy~\cite{Oren1982}, the scalar $\tau_k$ is
	computed at each iteration to satisfy the well-known secant equation in a least
	squares sense, cf. also \cref{rem_scalingfactorssecantequation}. 
	The secant equation for \cref{eq_choiceofseedmatrixgeneralversion} reads
	\begin{equation}\label{eq:Bk0Secant}
		y_k=B_{k+1}^{(0)} s_k=\tau_{k+1} s_k+S_{k+1} s_k
		\quad\iff\quad \tau_{k+1} s_k-z_k=0
		\text{ with }
		z_k:=y_k-S_{k+1} s_k.
	\end{equation}
	Note that $z_k$ appears in \cref{line_zk} of Algorithm~\ref{alg_hybrid}. 
	Due to \cref{eq:Bk0Secant} it seems reasonable to choose $\tau_{k+1}$ as the minimizer of $\norm{\tau s_k-z_k}$ or a related least squares problem. 
	Let us also mention that the use of cautious updates in Algorithm~\ref{alg_hybrid} makes it even more compelling to incorporate information 
	from $(s_k,y_k)$ in the seed matrix $B_{k+1}^{(0)}$. 
	Specifically, we observe that the cautious update in \cref{line_acceptanceofysforstorage} may prevent $(s_k,y_k)$ from entering the storage, 
	in which case the only remaining option for $(s_k,y_k)$ to influence $B_{k+1}$ is via $B_{k+1}^{(0)}$, cf. \cref{line_tlr}.

	We now discuss four choices for the scaling factor $\tau_{k+1}$, 
	two of which follow the BB strategies from \cref{def:tau:H}, one is an unbiased total least squares approach, and the
	last one is the geometric mean of the first two.
	We take an algebraic point of view to derive the scaling factors, 
	but note that they may also be derived through geometric and statistical arguments; see \cite{Draper1991} for details. 
	For $d_k=-H_k\nabla\CJ(x_k)$ to be a descent direction, 
	we want $H_k$ to be positive definite. 
	It is well known that this can be achieved by choosing $H_k^{(0)}$ positive definite and storing only those $(s_j,y_j)$ that satisfy $y_j^T s_j>0$. 
	However, to show convergence to stationary points we need to control the condition number of $H_k$ (or, equivalently, $B_k$). 
	To this end, we will confine $\tau_{k+1}$ to an interval $[\tau_{k+1}^{\min},\tau_{k+1}^{\max}]$, represented by $[\tau_{\min},\tau_{\max}]$ in the following definition.
	
	\begin{definition}\label{def:tau:B}
		Let $0\leq\tau_{\min}\leq\tau_{\max}\leq\infty$. We define 
		\begin{equation*}
			\proj:\R\rightarrow[\tau_{\min},\tau_{\max}], \qquad \proj(t):=\min\Bigl\{\max\{t,\tau_{\min}\},\tau_{\max}\Bigr\}.	
		\end{equation*}
		For $(s,z)\in\CX\times\CX$ with $s\neq 0$ let 
		$\rho:=z^T s$, $\lambda:=\frac{\norm{s}^2+\norm{z}^2-\sqrt{(\norm{s}^2-\norm{z}^2)^2+4\rho^2}}{2}$
		and
		\begin{equation*}
				\tau^\Bp:=\proj\biggl(\frac{\rho}{\norm{s}^{2}}\biggr),\qquad
				\tau^\GM:=\proj\biggl(\frac{\norm{z}}{\norm{s}}\biggr),\qquad
				\tau^\Bz:=\proj\biggl(\frac{\norm{z}^2}{\rho}\biggr),\qquad 
				\tau^\Bu:=\proj\biggl(\frac{\norm{z}^2-\lambda}{\rho}\biggr),
		\end{equation*}
		where $\tau^\Bz$ and $\tau^\Bu$ are only defined if $\rho\neq 0$.
	\end{definition}
	
	We note that for $\rho>0$ the scaling factors 
	$\tau^\Bp$, $\tau^\Bz$ and $\tau^\GM$ are the minimizers of the following least squares problems constrained by~$\tau_{\min}\leq\tau\leq\tau_{\max}$:
	$\tau^\Bp$ for~$\norm{\tau s-z}$, $\tau^\Bz$ for~$\norm{s-z/\tau}$ 
	and $\tau^\GM$ for~$\norm{\sqrt{\tau}s-z/\sqrt{\tau}}$; cf. also \cite[Section~2.1]{Dener2019}.
	The unbiased choice $\tau^\Bu$ is motivated in \cite{AM21}. 

	Next we relate the scaling factors to each other. We illustrate the relation for $\rho>0$ in \cref{fig:pz}.
	
	\begin{lemma}\label{lem_relationbetweendifferenttaus}
		For $\rho>0$ the scaling factors from \cref{def:tau:B} satisfy $0\leq\tau^\Bp\le \tau^\Bu\le \tau^\Bz$ and $0\leq\tau^\Bp\le\tau^\GM\le\tau^\Bz$.
		For $\rho\leq 0$ there holds $0\leq\tau_{\min}=\tau^\Bp\le \tau^\GM$.
	\end{lemma}
	
	\begin{proof}
		An application of the Cauchy-Schwarz inequality establishes the claim for $\rho\leq 0$.
		Let now $\rho>0$. 
		The Cauchy-Schwarz inequality implies $\tau^\Bp\le \tau^\Bu\le \tau^\Bz$.
		Since $\tau^\GM$ is the geometric mean of $\tau^\Bp$ and $\tau^\Bz$, 
		we have $\tau^\Bp\le\tau^\GM\le\tau^\Bz$.
		It is easy to confirm that $\lambda\geq 0$ and $\norm{z}^2-\lambda\geq 0$.
		It can also be checked that 
		$(\norm{z}^2-\lambda)(\norm{s}^2-\lambda)=\rho^2$.
		Taken together, we obtain $\tau^\Bp\le \tau^\Bu\le \tau^\Bz$.
	\end{proof}
	
	\begin{figure}[h]
		\centering
		\begin{tikzpicture}
			\begin{scope}[xshift=0]
				\coordinate (N) at (0,0);
				\coordinate (P) at (0:1);
				\coordinate (Z) at (45:2);
				\coordinate (Hy) at (45:{1/sqrt(2)});
				\coordinate (Hp) at (0:{4/(sqrt(2)*1)});
				
				\draw[-,thin]  (N) -- (0:3);
				\draw[->,thick] (N) -- (P) node[midway,below] {$s$};
				\draw[-] 		(N) -- (45:3);
				\draw[->,thick]	(N) -- (Z) node[midway,above] {$y$};
				
				\draw[dashed] 	(P) node[below] {$\hat\tau^\Hy$}-- (Hy)	 node[midway,above] {};
				\draw[dotted] 	(Z) -- (Hp) node[below] {$\hat\tau^\Hs$};
				\node at(1.5,-1) {Scaling factors from \cref{def:tau:H}};
			\end{scope}
			
			\begin{scope}[xshift=6cm]
				\pgfmathsetmacro{\ang}{45}
				\pgfmathsetmacro{\u}{3*cos(\ang)}
				\pgfmathsetmacro{\v}{3*sin(\ang)}
				\pgfmathsetmacro{\s}{\u/9}
				\coordinate (N) at (0,0);
				\coordinate (P) at (0:1);
				\coordinate (Z) at (\u,\v);
				\coordinate (Bp) at (\u,0);
				\coordinate (Bz) at (1/\s,0);
				\coordinate (Bu) at (0.9704/0.2414,0);
				\coordinate (GM) at (3,0);
				
				\draw[-,thin]   (N) -- (0:5);
				\draw[->,thick] (N) -- (P) node[below] {$s$};
				\draw[-] 		(N) -- (45:4);
				\draw[->,thick] (N) -- (Z) node[above] {$z$};
				
				\draw[dashed] 	(Z) -- (Bp) node[below] {$\tau^\Bp$};
				\draw[dotted] 	(Z) -- (Bz) node[below right] {$\tau^\Bz$};
				\draw[loosely dotted] (Z) -- (Bu) node[below] {$\tau^\Bu$};
				\draw[dash dot] (Z) -- (GM) node[below] {$\tau^\GM$};
				\node at(2,-1) {Scaling factors from \cref{def:tau:B}};
				
			\end{scope}
		\end{tikzpicture}
		
		\caption{Approximation strategies: 
			Left: BB step sizes for $H_k=\hat\tau_k I$; cf.~\cref{lem:tau}.
			Right: Scaling factors from \cref{def:tau:B} for $B_k=\tau_k I+S_k$;
			cf.~\cref{lem_relationbetweendifferenttaus}.}
		\label{fig:pz}
	\end{figure}

	\subsection{The two-loop recursion}\label{sec_tlrdisc}
	We assume the reader to be familiar with the two-loop recursion \cite[Algorithm~7.4]{Nocedal2006}. In it, the computation of $r=H_k^{(0)}q$ is required after the first loop, where $q$ has been determined during the first loop. 
	If $H_k^{(0)}$ is a scaled identity, $r$ is very cheap to obtain. 
	In \ref{alg_hybrid}, however, we must compute $r$ by solving $B_k^{(0)}r=q$, where $B_k^{(0)}=\tau_k I + S_k$.  
	Therefore, we require the choice of $S_k$ to be such that linear systems involving $B_k^{(0)}$ can be solved efficiently, at least approximately by an iterative method. For convex regularizer $\CS$ and the choice $S_k=\nabla^2\CS(x_k)$ this is satisfied if $\nabla^2\CS(x_k)$ is well-conditioned and sparse, which is often true for regularizers in inverse problems.
	In particular, these conditions are fulfilled for the real-world problems in the numerical experiments, so an iterative method with early stopping yields a good approximation of $r$. 
		
	\subsection{Cautious updating}\label{sec_cautiousupdating}
	Algorithm~\ref{alg_hybrid} uses \emph{cautious updating} \cite{LiFukucautiousupdates} both for the decision whether $(s_k,y_k)$ is stored and for the choice of $\tau_{k+1}$. The latter is realized by using $\tau_{\min}=\omega_{k+1}^l$ and $\tau_{\max}=\omega_{k+1}^u$, effectively safeguarding $\tau_{k+1}$ from becoming too small or too large in relation to $\nabla\CJ(x_{k+1})$. 
	This control over the condition number of $B_k$ is critical to prove, without convexity assumptions on $\CJ$, that 
	$\lim_{k\to\infty}\nabla\CJ(x_{k})=0$. Note that the safeguards $\omega_{k+1}^l$ and $\omega_{k+1}^u$ converge to zero, respectively, infinity. 
	Also observe that $(s_k,y_k)$ is only stored if it satisfies $y_k^T s_k > c_s\norm{s_k}^2$ for a small constant $c_s>0$.
	This is necessary to control the condition number of $B_k$. 
	Cautious updating has been used previously for the rank-two updates, e.g. in \cite{BJMT22,KS23}, 
	but using it for $\tau_k$ is new. 
	
	\subsection{A remark on quadratic regularizers}
	\label{sec:simpQuadReg}
	
	The computation of $z_k=y_k-S_{k+1} s_k$ in \eqref{eq:Bk0Secant} can be simplified if $\CS(x) = 0.5x^T A x+b^T x + c$ is quadratic and $S_k=\nabla^2\CS=A$ is used.
	We have 
	\begin{equation*}
	z_k = y_k-A s_k =\nabla \CJ(x_{k+1})-\nabla \CJ(x_{k})-A(x_{k+1}-x_{k})
	=\nabla \CD(x_{k+1})-\nabla \CD(x_{k}).
	\end{equation*}
	Computing $z_k$ as the difference between two consecutive gradients of the data
	term avoids the computation of $A s_k$ for the price of storing the two gradients $\nabla\CD(x_{k+1})$ and $\nabla\CD(x_k)$.
	
	\subsection{Comparison to other structured \LBFGS~methods}\label{sec_difftoPetra}
	
	In the introduction we briefly described how Algorithm~\ref{alg_hybrid} differs from existing structured \LBFGS~methods.
	We now comment on this in more detail for the recently proposed structured \LBFGS~methods \PetraM~and \PetraP~from \cite{BDLP22}.

	Let us first observe that in general the different methods use different operators $B_k$. 
	The operator $B_{k+1}$ in Algorithm~\ref{alg_hybrid} satisfies the secant equation 
	$B_{k+1} s_k = y_k = \nabla\CJ(x_{k+1})-\nabla\CJ(x_k)$, 
	while the operators in \cite{BDLP22} satisfy, in our notation, the secant equations
	$[B_{k+1}-\nabla^2\CS(x_k)]s_k = \nabla\CD(x_{k+1})-\nabla\CD(x_k)$ (\PetraM),
	respectively, $[B_{k+1}-\nabla^2\CS(x_{k+1})]s_k = \nabla\CD(x_{k+1})-\nabla\CD(x_k)$ (\PetraP). 
	This implies that these operators are generally different from each other. 
	On the other hand, all $B_{k+1}$ agree if $\CS$ is quadratic, $S_{k+1}=\nabla^2\CS$ and the same $\tau_{k+1}$ is used for the identity. 
	For the secant equations this is obvious from $\nabla^2\CS(x_{k+1})s_k=\nabla^2\CS(x_{k})s_k = \nabla\CS(x_{k+1})-\nabla\CS(x_{k})$.
		
	Another important difference is that the algorithms in \cite{BDLP22} use \emph{compact representations}~\cite{BNS94} of the limited memory matrices, whereas Algorithm~\ref{alg_hybrid} is based on a two-loop recursion. While the \PetraM~updates can also be implemented with a two-loop recursion, this is not true for \PetraP. 
	This seriously hampers the numerical performance as utilizing the seed matrix $B_k^{(0)}=\tau_k I+S_k$ in the two-loop recursion requires to solve one linear system of type $B_k^{(0)}x = b$ per \LBFGS~iteration, whereas the compact representation requires solving $2\ell+1$ such linear systems, 
	$\ell$ being the number of update pairs. 
	Consequently, \PetraP~turns out to be substantially slower than the other two methods. 
		
	In summary, there are meaningful differences between the structured \LBFGS~methods of \cite{BDLP22} and Algorithm~\ref{alg_hybrid}. 
	We compare their numerical performance 
	in \cref{sec:experiments}. 
	
	\subsection{A robust choice of the scaling factor \texorpdfstring{$\tau_k$}{tau_k} for image registration}\label{sec:adaptInit}
	
	As already alluded to, the choice of $\tau_k$ in Algorithm~\ref{alg_hybrid} can have a significant impact on the performance. 
	We now discuss 
	a robust choice of this parameter for image registration problems. 
	The result is summarized in Algorithm~\ref{alg_adaptivechoiceoftau} and we include it in the numerical experiments.
	
	Let us start with the case $\rho=z^T s \leq 0$, cf. \cref{line_choicetauk2}. 
	In that case, we always choose $\tau_{k+1} = \tau^\GM$ since this proved most effective in the numerical experiments. It remains to discuss the case $\rho>0$.
	
	One possibility for that case is to consistently select the same scaling factor from \cref{def:tau:B}, e.g., to use $\tau_{k+1}=\tau^\Bp$ for all $k$. 
	In the experiments on image registration problems we observe that most often either $\tau^\Bp$ or $\tau^\GM$ produce the lowest run-time while $\tau^\Bz$ achieves the highest accuracy. However, sometimes $\tau^\Bp$ and $\tau^\GM$ are quite inaccurate in comparison to $\tau^\Bz$ and sometimes convergence for $\tau^\Bz$ is extremely slow. 
	To achieve a better balance we choose $\tau_{k+1}$ as a weighted geometric mean of $\tau^\Bp$, $\tau^\GM$ and $\tau^\Bz$, where the weights $\wbp$,  $\wgm$ and $\wbz$ satisfy $\wbp + \wgm + \wbz = 1$, but are allowed to change at every iteration. 
	In the first iterations, we make $\wbp$ the largest weight. 
	As the iterations progress, we gradually shift towards $\tau^\GM$ by decreasing $\wbp$ while increasing $\wgm$. 
	When it is estimated that the algorithm is somewhat close to termination, we shift towards $\tau^\Bz$.
	The velocity with which the scaling factor moves to $\tau^\GM$, respectively, $\tau^\Bz$ depends on the number of line searches.
	Specifically, if the number of line searches is large this indicates that $d_k$ is poorly scaled, so we increase the rate with which we move to the larger
	scaling factor $\tau^\GM$, resp., $\tau^\Bz$ (recall from \cref{lem_relationbetweendifferenttaus} that $\tau^{\Bp}\leq\tau^\GM\leq\tau^\Bz$ if $\rho>0$).

\begin{algorithm2e}[h!]
\SetAlgoRefName{ADAP}
\DontPrintSemicolon
\caption{Adaptive choice of $\tau_{k+1}$ in \cref{line_choicetauk1}--\cref{line_choicetauk2} of Algorithm~\ref{alg_hybrid}}
\label{alg_adaptivechoiceoftau} 
\KwIn{$\Delta_0\in[0.5,1]$, $\Delta_1\in [0,1]$, $0 < \eps_1 < \eps_0 << 1$, $0 < \eta_0 < \eta_1 < \eta_2 < 1$, $\beta\in(0,1)$,$\wbp_{-1}:=0$}
Compute $\tau^\Bp$, $\tau^\GM$ and $\tau^\Bz$ from \cref{def:tau:B} with $(s,z)=(s_k,z_k)$, $\tau_{\min}=\omega_{k+1}^l$, $\tau_{\max}=\omega_{k+1}^u$\\
Let $\nu\geq 1$ be the number of line search steps taken in \cref{line_nols} of Algorithm~\ref{alg_hybrid}\\
\lIf{$k = 0$} 
{let $\wbp_k := \Delta_0$, $\wgm_k := 1 - \wbp_k$, $\wbz_k := 0$ \tcp*[f]{prefer $\tau^\Bp$ at start}}
\lElseIf{$|\CJ(x_{k+1})-\CJ(x_k)| \leq \eps_1 |\CJ(x_k)|$\label{line_eps1}}{let $\alpha := \eta_2$ \tcp*[f]{large progress}}
\lElseIf{$|\CJ(x_{k+1})-\CJ(x_k)| \leq \eps_0 |\CJ(x_k)|$}{let $\alpha := \eta_1$ \label{line_eta1}\tcp*[f]{medium progress}}
\lElse{let $\alpha := \eta_0$\tcp*[f]{small progress}}
\lIf{$\wbp_{k-1} > 0$} 
{let $\wbp_k := \max\{\wbp_{k-1} - \alpha \nu,0\}$, $\wgm_k := 1 - \wbp_k$\label{line_eta0}\tcp*[f]{shift weight from $\tau^\Bp$ to $\tau^\GM$}}
\lIf{$\wgm_{k-1} \geq 1$ or $\wbz_{k-1}>0$}{let $\wgm_k := \max\{\wgm_{k-1} - \beta \nu, \Delta_1\}$, $\wbz_k := 1 - \wgm_k$\label{line_beta}\tcp*[f]{shift $\tau^\GM$ to $\tau^\Bz$}}
Let $\tau_{k+1}:=(\tau^\Bp)^{\wbp_k}(\tau^\GM)^{\wgm_k}(\tau^\Bz)^{\wbz_k}$ \tcp*{Define $\tau_{k+1}$ based on obtained weights}
\lIf{$z_k^T s_k \leq 0$}{let $\tau_{k+1}:=\tau^\GM$\label{line_stz_1}\tcp*[f]{prefer $\tau^\GM$ to cut off}}
\end{algorithm2e}
	

\section{Convergence analysis}\label{sec_convana}
	
	This section is devoted to the convergence of Algorithm~\ref{alg_hybrid}. 
	After some preliminaries we show global convergence in \cref{sec_globconv}
	and linear convergence in \cref{sec_linconv}.
	The analysis takes place in an infinite dimensional Hilbert space $\CX$, which is new for \LBFGS.	
	We will use the notation $\LinPDX:=\{M\in\Lin(\CX): \,M \text{ is symmetric positive definite}\}$. 
		
	\subsection{Preliminaries}
	
	In this section we estimate $\norm{B_k}$ and $\norm{B_k^{-1}}$ and we show that Algorithm~\ref{alg_hybrid} is well-defined.
	
	\begin{lemma}\label{lem_generalresultonnormbounds}
		Let $M\in\N_0$ and $\kappa_1,\kappa_2>0$. For all $j\in\{1,\ldots,M\}$ let $(s_j,y_j)\in\CX\times\CX$ 
		satisfy 
		\begin{equation*}
			\frac{y_j^T s_j}{\norm{s_j}^2}\geq \frac{1}{\kappa_1}\qquad\text{ and }\qquad
			\frac{y_j^T s_j}{\norm{y_j}^2}\geq \frac{1}{\kappa_2}. 
		\end{equation*}
		Moreover, let $B^{(0)}\in\LinPDX$. Then the \LBFGS~update $B$, obtained from the recursion \cref{eq_defLBFGSupd1} with $B_k^{(0)}=B^{(0)}$, $m=1$ and $\ell=M$, 
		satisfies $B\in\LinPDX$ as well as the estimates
		\begin{equation}\label{eq_estlargestEV}
			\norm{B}\leq\norm{B^{(0)}} + M \kappa_2,
		\end{equation}
		and 
		\begin{equation}\label{eq_estsmallestEV}
			\norm{B^{-1}}\leq 5^M\max\bigl\{1,\norm{(B^{(0)})^{-1}}\bigr\}\max\bigl\{1,\kappa_1^M,(\kappa_1 \kappa_2)^M\bigr\}.
		\end{equation}
	\end{lemma}
	
	\begin{proof}
		It is clear from the definition, cf. \cref{eq_defU}, that $\upd(s,y,C)$ is self-adjoint if $C\in\Lin(\CX)$ is self-adjoint and $y^T s\neq 0$. 
		Inductively, this implies that $B$ is self-adjoint.
		Concerning \cref{eq_estlargestEV} note that the \LBFGS~formula yields by induction over $j$ that for any $v$ with $\norm{v}=1$ and $0\leq j\leq M-1$ there holds
		\begin{equation*}
			v^T B_{j+1} v = v^T B_j v - \frac{(s_j^T B_j v)^2}{s_j B_js_j} + \frac{(y_j^T v)^2}{y_j^T s_j}
			\leq v^T B_j v + \frac{\norm{y_j}^2\norm{v}^2}{y_j^T s_j},
		\end{equation*}
		where $B_0:=B^{(0)}$ and $B=B_M$. Here, we have also used that each $B_j$ is positive definite; for $j=0$ this holds by assumption, while for $j>0$ this follows by rewriting $v=p+w$, where $p$ is the orthogonal projection of $v$ onto $s_j$ with respect to the scalar product $(v',v'')_j:=(v',B_j v'')$, and $w=v-p$.
		Hence, 
		\begin{equation*}
			\norm{B_{j+1}}\leq \norm{B_j} + \kappa_2
		\end{equation*}
		for all $0\leq j\leq M-1$, which implies \cref{eq_estlargestEV}.
		The proof of \cref{eq_estsmallestEV} uses 
		that the inverse $H:=B^{-1}$ equals $H_M$ in the recursion
		\begin{equation*}
			H_{j+1} = V_j^T H_{j} V_j + \rho_j s_j s_j^T, \qquad j=0,1,\ldots,M-1,
		\end{equation*}
		where $H_0:=(B^{(0)})^{-1}$, $V_j:=I-\rho_j y_j s_j^T$, and $\rho_j := (y_j^T s_j)^{-1}$; cf. \cite[(6.17)]{Nocedal2006}.	
		Since it is similarly elementary as the proof of \cref{eq_estlargestEV}, we omit it.
	\end{proof}	
	
	The following assumption ensures that Algorithm~\ref{alg_hybrid} is well-defined.
	
	\begin{assumption}\label{ass_basic}
		\phantom{to enforce linebreak}
		\begin{enumerate}
			\item[1)] The function $\CJ:\CX\rightarrow\R$ is continuously differentiable and bounded below. 
			\item[2)] The step size $\alpha_k$ in \ref{alg_hybrid} consistently satisfies the Armijo condition and is computed by backtracking, cf. \cref{eq_armijocond}, or it consistently satisfies the Wolfe--Powell conditions, cf. \cref{eq_WolfePowellcond}.
			\item[3)] The value $c_0=0$ is only chosen in Algorithm~\ref{alg_hybrid} if with this choice there holds $B_k\in\LinPDX$ for all $k$ (which is, for instance, the case if $S_k\in\LinPDX$ for all $k$).
		\end{enumerate}
	\end{assumption}
	
	\begin{lemma}\label{lem_welldef}
		If \cref{ass_basic} holds, then Algorithm~\ref{alg_hybrid} is well-defined and the sequence $(\CJ(x_k))$ is strictly monotonically decreasing and convergent. 
	\end{lemma}		
	
	\begin{proof}
		Since $B_k\in\LinPDX$ for all $k$, $H_k$ is well-defined, so $d_k$ exists. 
		This also yields $\nabla\CJ(x_k)^T d_k = -d_k^T B_k d_k < 0$, hence $d_k$ is a descent direction for all $k$. 
		In turn, both Armijo step sizes and Wolfe--Powell step sizes exist for all $k$, the latter because 
		$\CJ$ is bounded below. 
		Due to \cref{lem_relationbetweendifferenttaus} and $\omega_{k+1}^l\leq\omega_{k+1}^u$ the interval $[\tau^\Bp,\tau^\Bz]$, respectively, $[\tau^\Bp,\tau^\GM]$
		is nonempty for all $k$. 
		Together, it follows that the entire algorithm is well-defined.
		The step sizes ensure $\CJ(x_{k+1})<\CJ(x_k)$ for all $k$, 
		so $(\CJ(x_k))$ is strictly monotonically decreasing. As $\CJ$ is bounded below, 
		$(\CJ(x_k))$ converges.
	\end{proof}	
	
	\subsection{Global convergence of Algorithm~\ref{alg_hybrid}}\label{sec_globconv}
	
	We associate to $x_0$ in Algorithm~\ref{alg_hybrid} the level set 
	\begin{equation*}
		\Omega:=\Bigl\{x\in\CX: \; \CJ(x)\leq\CJ(x_0)\Bigr\}.
	\end{equation*}
	Moreover, for $\delta>0$ we define the following neighborhood of $\Omega$,
	\begin{equation*}
		\Omega_\delta:=\Omega+\Ballop{\delta}(0)=\Bigl\{x\in\CX: \enspace \exists \hat x\in\Omega: \,\norm{x-\hat x}<\delta\Bigr\}.
	\end{equation*}
	
	We establish the global convergence of Algorithm~\ref{alg_hybrid} under the following assumption.
	
	\begin{assumption}\label{ass_globconv}	
		\phantom{to create linebreak}
		\begin{enumerate}
			\item[1)] The function $\CJ:\CX\rightarrow\R$ is continuously differentiable and bounded below.
			\item[2)] The gradient of $\CJ$ is Lipschitz continuous in $\Omega$ with constant $L>0$, i.e., there holds
			$\norm{\nabla\CJ(x)-\nabla\CJ(\hat x)}\leq L\norm{x-\hat x}$ for all $x,\hat x\in\Omega$.
			\item[3)] The sequence $(\norm{S_k})$ in Algorithm~\ref{alg_hybrid} is bounded. 
			\item[4)] The step size $\alpha_k$ in Algorithm~\ref{alg_hybrid} consistently satisfies the Armijo condition and is computed by backtracking, cf. \cref{eq_armijocond}, or it consistently satisfies the Wolfe--Powell conditions, cf. \cref{eq_WolfePowellcond}.
			If the Armijo condition with backtracking is used for step size selection, we suppose in addition that there is $\delta>0$ such that $\CJ$ or $\nabla\CJ$ is uniformly continuous in $\Omega_\delta$. 
			\item[5)] The value $c_0=0$ is only chosen in Algorithm~\ref{alg_hybrid} if 
			 with this choice there holds $\sup_k\norm{(B_k^{(0)})^{-1}}<\infty$
			 (which is, for instance, the case if $(\norm{S_k^{-1}})$ is bounded). 
			\item[6)] $C_0=\infty$ is only chosen in Algorithm~\ref{alg_hybrid} if any of the following holds:
			\begin{itemize}
				\item The interval $[\tau^\Bp,\tau^\Bz]$ in \cref{line_choicetauk1} of Algorithm~\ref{alg_hybrid} is replaced by  $[\tau^\Bp,\tau^\GM]$.
				\item $\CJ$ is twice continuously differentiable, $\overline{G_k}:=\overline{\nabla^2\CJ_k}-S_{k+1}$ is symmetric positive semi-definite for all $k$, and 
				$(\norm{\overline{G_k}})$ is bounded.
				For quadratic $\CS$ and $S_k=\nabla^2\CS(x_k)$ for all $k$, we can also replace
				$\overline{G_k}$ in the preceding sentence by 
				$\overline{\nabla^2\CD_k}:=\int_0^1 \nabla^2\CD(x_k+t s_k)\,\mathrm{d}t$.
			\end{itemize}
		\end{enumerate}		
	\end{assumption}
	
	\begin{remark}
		The sequence $(\norm{S_k})$ is for instance bounded if we select $S_k=\nabla^2\CS(x_k)$ for all $k$, $(x_k)$ is bounded and $\nabla^2\CS$ is Hölder continuous in $\Omega$. The sequence $(\norm{(B_k^{(0)})^{-1}})$ is for instance bounded if
		we select $S_k=\nabla^2\CS(x_k)$ for all $k$ and the regularizer $\CS$ is strongly convex.
		Note that $c_0=0$ implies $\tau_{\min}=\omega_{k+1}^l=0$ for all $k$, while $C_0=\infty$ implies $\tau_{\max}=\omega_{k+1}^u=\infty$ for all $k$. That is, for $c_0=0$, respectively, $C_0=\infty$, the lower safeguard $\tau_{\min}$ is zero, resp., the upper safeguard $\tau_{\max}$ is irrelevant when selecting $\tau_{k+1}$, just as in classical \LBFGS.
		Observe in this context that if $\CS\equiv 0$ and $\CD$ is a strongly convex $C^2$ function with Lipschitz continuous gradient in $\Omega$, then 
		\cref{ass_globconv} holds with $c_0=0$ and $C_0=\infty$ and we can use $S_k=0$ for all $k$. In this case we recover classical~\LBFGS, including common choices of $\tau_{k+1}$ (for the latter note that $z_k=y_k$ for all $k$ and that $\proj$ in \cref{def:tau:B} disappears).
		Also note that \cref{ass_globconv} implies \cref{ass_basic}, with 3) following from \cref{lem_generalresultonnormbounds}.
	\end{remark}
	
	We now prove fundamental estimates for $\norm{B_k}$, $\norm{B_k^{-1}}$, $\norm{\nabla\CJ(x_k)}$ and $\norm{d_k}$.
	
	\begin{lemma}\label{lem_BkandBkinversebounded}
		Let \cref{ass_globconv} hold and let $(x_k)$ be generated by Algorithm~\ref{alg_hybrid}. 
		Then there is a constant $C>0$ such that
		\begin{equation*}
			\norm{B_k}+\norm{B_k^{-1}}\leq C \max\bigl\{1,\norm{\nabla\CJ(x_k)}^{-c_2}\bigr\}
		\end{equation*}
		is satisfied for all $k\geq 0$, where $c_2$ is the constant from Algorithm~\ref{alg_hybrid}. 
	\end{lemma}	
	
	\begin{proof}
		In view of \cref{eq_estlargestEV}, respectively, \cref{eq_estsmallestEV} and the acceptance criterion $y_k^T s_k>c_s \norm{s_k}^2$ in \cref{line_acceptanceofysforstorage} of Algorithm~\ref{alg_hybrid}
		we have 
		\begin{equation*}
			\norm{B_k}\leq\norm{B_k^{(0)}} + \frac{L^2\ell}{c_s}
		\end{equation*}
		and
		\begin{equation}\label{eq_estBkinv}
			\norm{B_k^{-1}}\leq 5^\ell\max\bigl\{1,\norm{(B_k^{(0)})^{-1}}\bigr\}\max\bigl\{1,c_s^{-\ell},\left(L/c_s\right)^{2\ell}\bigr\}.
		\end{equation}
		It remains to estimate $\norm{B_k^{(0)}}$ and $\norm{(B_k^{(0)})^{-1}}$.
		
		\textbf{An estimate for $\norm{B_k^{(0)}}$}\\
		Since $\norm{B_k^{(0)}}=\tau_k + \norm{S_k} \leq \omega_k^u + \norm{S_k}$ and since $C_S:=\sup_k\norm{S_k}<\infty$ by assumption, the definition of $\omega_k^u$ in Algorithm~\ref{alg_hybrid} implies 
		$\norm{B_k}\leq C\max\{1,\norm{\nabla\CJ(x_k)}^{-c_2}\}$ if $C_0<\infty$. If $C_0=\infty$, then there are two cases, cf. \cref{ass_globconv}~6).
		In the first case, all $\tau_k$ are restricted to $[\tau^\Bp,\tau^\GM]$, so
		it suffices to show that $\tau^\GM$ is bounded from above independently of $k$.
		Indeed, we have
		\begin{equation*}
			\tau^{\GM}\leq\frac{\norm{z_{k-1}}}{\norm{s_{k-1}}}
			\leq \frac{\norm{y_{k-1}} + \norm{S_k s_{k-1}}}{\norm{s_{k-1}}}
			\leq L+\norm{S_k}\leq L+C_S.
		\end{equation*}
		In the second case and for $\rho_{k-1}\leq 0$ we have $\tau_k\leq\tau^\GM$, so we can reuse the preceding estimate.
		For $\rho_{k-1}>0$ we show that $\tau^\Bz$ is bounded from above independently of $k$, which by \cref{lem_relationbetweendifferenttaus} and $\tau_k\leq\tau^\Bz$ implies that $\tau_k$ is bounded from above for $\rho_{k-1}>0$, too.
		Using $z_{k-1}=\overline{G_{k-1}} s_{k-1}$ we find
		\begin{equation*}
			\tau^{\Bz}=\frac{\norm{z_{k-1}}^2}{\rho_{k-1}}
			= \frac{z_{k-1}^T z_{k-1}}{z_{k-1}^T s_{k-1}}
			= \frac{s_{k-1}^T \overline{G_{k-1}}^2 s_{k-1}}{s_{k-1}^T \overline{G_{k-1}} s_{k-1}}
			\leq \norm{\overline{G_{k-1}}},
		\end{equation*}
		so the boundedness of $(\tau_k)$ follows by assumption.\\
		For later reference we note that 
		\begin{equation}\label{eq_Bkbounded}
			\text{if one of the statements in \cref{ass_globconv}~6) holds, then $\sup_k\,\norm{B_k}<\infty$.} 
		\end{equation}
		\textbf{An estimate for $\norm{(B_k^{(0)})^{-1}}$}\\		
		If $c_0=0$, then $(\norm{(B_k^{(0)})^{-1}})$ is bounded by assumption.
		It remains to consider the case $c_0>0$, where $\omega_k^l>0$ for all $k$. 
		For every $v$ with $\norm{v}=1$ we have $v^T B_k^{(0)} v = \tau_k + v^T S_k v\geq \tau_k$, which yields
		$\norm{(B_k^{(0)})^{-1}}\leq 1/\tau_k \leq 1/\omega_k^l$.
		Using the definition of $\omega_k^l$ and inserting into \cref{eq_estBkinv} yields the desired estimate
		$\norm{B_k^{-1}}\leq C\max\{1,\norm{\nabla\CJ(x_k)}^{-c_2}\}$. 
	\end{proof}	
	
	\begin{corollary}\label{cor_BkandBkinversebounded}
		Let \cref{ass_globconv} hold and let $(x_k)$ be generated by Algorithm~\ref{alg_hybrid}. 
		Then there are constants $c,C>0$ such that
		\begin{equation}\label{eq_boundonlengthofgradientsandsteps}
			c\cdot\min\Bigl\{1,\norm{\nabla\CJ(x_k)}^{c_2}\Bigr\}\leq \frac{\norm{d_k}}{\norm{\nabla\CJ(x_k)}}\leq C\cdot\max\Bigl\{1,\norm{\nabla\CJ(x_k)}^{-c_2}\Bigr\}
		\end{equation}
		as well as 
		\begin{equation*}
			\frac{\lvert\nabla\CJ(x_k)^T d_k\rvert}{\norm{d_k}}  \geq C^{-1}\norm{\nabla\CJ(x_k)}\min\Bigl\{1,\norm{\nabla\CJ(x_k)}^{2c_2}\Bigr\}
		\end{equation*}
		for all $k\geq 0$. 
	\end{corollary}
	
	\begin{proof}
		Since $B_k d_k=-\nabla\CJ(x_k)$, the estimates \cref{eq_boundonlengthofgradientsandsteps} readily follow from those for $\norm{B_k}$ and $\norm{B_k^{-1}}$.
		Concerning the final inequality we have
		\begin{equation*}
			\begin{split}
				-\nabla\CJ(x_k)^T d_k = d_k^T B_k d_k 
				&\geq \norm{d_k}^2\norm{B_k^{-1}}^{-1}\\
				&\geq c\norm{d_k} \norm{\nabla\CJ(x_k)}\min\Bigl\{1,\norm{\nabla\CJ(x_k)}^{2c_2}\Bigr\},
			\end{split}
		\end{equation*}
		where we used the estimate for $\norm{B_k^{-1}}$ and the first inequality of \cref{eq_boundonlengthofgradientsandsteps}.
	\end{proof}	
	
	As a main result we now prove that Algorithm~\ref{alg_hybrid} is globally convergent in the sense that $\lim_{k\to\infty}\,\norm{\nabla\CJ(x_k)} = 0$, without convexity of the objective.
	Note that this result also applies to the unregularized case $\CS\equiv 0$. 
	In that case, \ref{alg_hybrid} reduces to an \emph{unstructured} \LBFGS-type method.
	Even for such methods there are only few works available that show $\lim_{k\to\infty}\,\norm{\nabla\CJ(x_k)} = 0$ for non-convex objective functions,
	cf. the discussion in \cref{sec_introconvana}.
	In addition, the following result is the first to demonstrate global convergence of a \emph{structured} \LBFGS-type method.
		
	\begin{theorem}\label{thm_globconv}
		Let \cref{ass_globconv} hold. Then:
		\begin{enumerate}
			\item[1)] If Algorithm~\ref{alg_hybrid} is applied with $\tl=0$, then it either terminates after finitely many iterations 
			with an $x_k$ that satisfies $\nabla\CJ(x_k)=0$ or it generates a sequence $(x_k)$ such that 
			\begin{equation}\label{eq_liminfiszero}
				\lim_{k\to\infty}\,\norm{\nabla\CJ(x_k)} = 0.
			\end{equation}
			In particular, every cluster point of $(x_k)$ is stationary.
			\item[2)] If Algorithm~\ref{alg_hybrid} is applied with $\tl>0$, then it terminates after finitely many iterations
			with an $x_k$ that satisfies $\norm{\nabla\CJ(x_k)}\leq\tl$. 
		\end{enumerate}
	\end{theorem}
	
	\begin{proof}
		It is clear that 2) follows from 1). 
		Thus, we assume $\tl=0$ in the remainder of the proof.
		We recall from \cref{lem_welldef} that Algorithm~\ref{alg_hybrid} is well-defined, so if it does not terminate 
		with an $x_k$ that satisfies $\norm{\nabla\CJ(x_k)}=0$, it generates an infinite sequence $(x_k)$.
		If $(x_k)$ satisfies \cref{eq_liminfiszero}, then by continuity of $\CJ$ it follows that every cluster point $\xopt$ of $(x_k)$ satisfies $\nabla\CJ(\xopt)=0$. Hence, it only remains to establish \cref{eq_liminfiszero}. 
		We first prove this for Armijo step sizes.\\
		\textbf{A) Proof for Armijo step sizes}\\
		Suppose that \cref{eq_liminfiszero} were false. Then there exist $\epsilon'>0$ and a subsequence $(x_k)_K$ of $(x_k)$ such that 
		\begin{equation}\label{eq_contrassgradawayfromzero}
			\norm{\nabla\CJ(x_k)}\geq\epsilon' \qquad\forall k\in K.
		\end{equation}
		From $-\nabla\CJ(x_k)^T d_k = d_k^T B_k d_k$ and the Armijo condition we infer
		\begin{equation*}
			\sigma\sum_{k\in K}\alpha_k\frac{\norm{d_k}^2}{\norm{B_k^{-1}}}
			\leq -\sigma\sum_{k=0}^{\infty}\alpha_k\nabla\CJ(x_k)^T d_k
			\leq \sum_{k=0}^\infty \left[\CJ(x_{k})-\CJ(x_{k+1})\right] 
			<\infty.
		\end{equation*}
		From \cref{lem_BkandBkinversebounded} and \cref{eq_contrassgradawayfromzero} we obtain 
		$\sup_{k\in K}\norm{B_k^{-1}}<\infty$, which yields $\sum_{k\in K}\alpha_k\norm{d_k}^2<\infty$.
		From the first inequality in \cref{eq_boundonlengthofgradientsandsteps} and \cref{eq_contrassgradawayfromzero} we infer that 
		\begin{equation}\label{eq_dkboundedawayfromzero}
			\exists c>0: \quad \norm{d_k}\geq c \quad\forall k\in K.
		\end{equation}
		Hence, $\sum_{k\in K}\alpha_k\norm{d_k}^2<\infty$ implies $\sum_{k\in K}\alpha_k<\infty$.
		By considering separately $k\in K_1:=\{k\in K:\norm{d_k}\leq 1\}$ and $k\in K_2:=K\setminus K_1$,
		we deduce from $\sum_{k\in K}\alpha_k\norm{d_k}^2<\infty$ and $\sum_{k\in K}\alpha_k<\infty$
		that $\sum_{k\in K}\alpha_k\norm{d_k}<\infty$. 
		This shows 
		\begin{equation}\label{eq_yal}
			\lim_{K\ni k\to\infty}\alpha_k = \lim_{K\ni k\to\infty}\alpha_k\norm{d_k} = 0.
		\end{equation}
		In particular, the Armijo condition \cref{eq_armijocond} is violated for $\hat\alpha_k:=\alpha_k\beta^{-1}$ for all $k\in K$ large enough. Therefore, 
		\begin{equation*}
			-\hat\alpha_k\sigma\nabla\CJ(x_k)^T d_k > \CJ(x_k)-\CJ(x_k+\hat\alpha_k d_k)
			= -\hat\alpha_k\nabla\CJ(x_k+\theta_k\hat\alpha_k d_k)^T d_k
		\end{equation*}
		for all these $k$ and $\theta_k\in(0,1)$. Multiplying by $-\hat\alpha_k$ we obtain 
		\begin{equation*}
			(\sigma-1)\nabla\CJ(x_k)^T d_k + \left[\nabla\CJ(x_k) - \nabla\CJ(x_k+\theta_k\hat\alpha_k d_k)\right]^T d_k<0.
		\end{equation*}
		Due to $\sigma<1$ and $-\nabla\CJ(x_k)^T d_k = d_k^T B_k d_k$ there holds for all $k\in K$ sufficiently large 
		\begin{equation}\label{eq_qssi}
			(1-\sigma)\frac{\norm{d_k}^2}{\norm{B_k^{-1}}} < 
			\Bigl\lvert\left[\nabla\CJ(x_k) - \nabla\CJ(x_k+\theta_k\hat\alpha_k d_k)\right]^T d_k\Bigr\rvert.
		\end{equation}
		As $\theta_k\in(0,1)$, 
		it follows from \cref{eq_yal} that $\theta_k\hat\alpha_k\norm{d_k}\to 0$ for $K\ni k\to\infty$.
		For the remainder of the proof of A), we have to distinguish two cases.\\
		\textbf{Case I: The uniform continuity assumption holds for $\mathbf{\CJ}$}\\
		If $\CJ$ is uniformly continuous in $\Omega_\delta$, then we 
		conclude that for all $k\in K$ sufficiently large there holds
		$x_k+\theta_k\hat\alpha_k\norm{d_k}\in \Omega$, so the Lipschitz continuity of $\nabla\CJ$ in $\Omega$ entails
		\begin{equation*}
			\Bigl\lvert\left[\nabla\CJ(x_k) - \nabla\CJ(x_k+\theta_k\hat\alpha_k d_k)\right]^T d_k\Bigr\rvert
			< L\hat\alpha_k\norm{d_k}^2
		\end{equation*}
		for all these $k$. 
		Combining this with \cref{eq_qssi} we have
		\begin{equation*}
			1-\sigma< L\hat\alpha_k\norm{B_k^{-1}}.
		\end{equation*}
		For $K\ni k\to\infty$ we find, due to $\sup_{k\in K}\norm{B_k^{-1}}<\infty$ and \cref{eq_yal}, that $0<1-\sigma\leq 0$. This contradiction concludes the proof of Case~I for Armijo step sizes.\\
		\textbf{Case II: The uniform continuity assumption holds for $\mathbf{\nabla\CJ}$}\\
		If $\nabla\CJ$ is uniformly continuous in $\Omega_\delta$, we infer that for all $k\in K$ sufficiently large there holds
		\begin{equation*}
			\Bigl\lvert\left[\nabla\CJ(x_k) - \nabla\CJ(x_k+\theta_k\hat\alpha_k d_k)\right]^T d_k\Bigr\rvert
			\leq \omega_k\norm{d_k},
		\end{equation*}
		where $\omega_k\to 0$ for $K\ni k\to\infty$. 
		Combining this with \cref{eq_qssi} we have
		\begin{equation*}
			1-\sigma < \omega_k\frac{\norm{B_k^{-1}}}{\norm{d_k}}
		\end{equation*}
		for all $k\in K$ sufficiently large.
		For $K\ni k\to\infty$ we find, due to $\sup_{k\in K}\norm{B_k^{-1}}<\infty$, \cref{eq_dkboundedawayfromzero} and $\omega_k\to 0$, that $0<1-\sigma\leq 0$. This contradiction concludes the proof of Case~II.\\
		\textbf{B) Proof for Wolfe--Powell step sizes}\\
		From \cite[Thm.~3.2]{Nocedal2006} we know that under \cref{ass_globconv} the Wolfe--Powell step sizes
		satisfy 
		$\CJ(x_{k+1})\leq \CJ(x_k) - c\left(\frac{\nabla\CJ(x_k)^T d_k}{\norm{d_k}}\right)^2$ for all $k$ and a constant $c>0$. Hence,
		\begin{equation*}
			c\sum_{k=0}^{\infty}\left(\frac{\nabla\CJ(x_k)^T d_k}{\norm{d_k}}\right)^2 
			\leq \sum_{k=0}^\infty \left[\CJ(x_{k})-\CJ(x_{k+1})\right] 
			<\infty.
		\end{equation*}
		Due to \cref{cor_BkandBkinversebounded} there is a constant $\hat c>0$ such that 
		\begin{equation*}
			\hat c^2\sum_{k=0}^\infty \norm{\nabla\CJ(x_k)}^2\min\Bigl\{1,\norm{\nabla\CJ(x_k)}^{4c_2}\Bigr\} 
			\leq \sum_{k=0}^{\infty} \left(\frac{\nabla\CJ(x_k)^T d_k}{\norm{d_k}}\right)^2<\infty.
		\end{equation*}
		This implies that $\nabla\CJ(x_k)\to 0$ for $k\to\infty$.
	\end{proof}	
	
	\begin{remark}\label{rem_unifcontfindim}
			If $\CX$ is finite dimensional and $(x_k)$ is bounded, then the proof for Armijo step sizes can be modified in such a way that the uniform continuity in \cref{ass_globconv}~4) is not required.
	\end{remark}	
	
		\subsection{Rate of convergence of Algorithm~\ref{alg_hybrid}}\label{sec_linconv}
		
		The convergence of the classical \LBFGS~method is q-linear for the objective and r-linear 
		for the iterates under strong convexity of $\CJ$ in the level set, cf. \cite{Liu1989}. 
		For non-convex objectives, we are only aware of sublinear rates \cite{WMGL17,BJMT22}.		
		We prove two results. First we show under a \KL-type inequality, which is weaker than local strong convexity, that the objective converges q-linearly and the iterates and their gradients converge 
		r-linearly. In the second result we prove that the same type of convergence holds if there is a cluster point in whose neighborhood $\CJ$ is strongly convex. 
		Both results rely on the following assumption.
		
		\begin{assumption}\label{ass_linconv}
			\item[1)] \Cref{ass_globconv} holds.
			\item[2)] Algorithm~\ref{alg_hybrid} is applied with $\tl=0$ and 
			generates an infinite sequence $(x_k)$. 
			\item[3)] The sequences $(\norm{B_k})$ and $(\norm{B_k^{-1}})$ are bounded.
			\item[4)] If the Armijo condition with backtracking is used for step size selection in Algorithm~\ref{alg_hybrid}, we 
			there is $\delta>0$ such that $\CJ$ is uniformly continuous in $\Omega_\delta$ or $\nabla\CJ$ is Lipschitz continuous in $\Omega_\delta$. 
		\end{assumption}

		\begin{remark}
			The boundedness assumption~3) 
			is easy to satisfy in the setting of this paper. Specifically, we recall that according to \cref{eq_Bkbounded} the boundedness of $(\norm{B_k})$ is ensured if at least one of the two statements in \cref{ass_globconv}~6) holds. Notably, the first of those statements only limits the choice of the scaling parameter $\tau_k$. 
			The boundedness of $(\norm{B_k^{-1}})$ is, for instance, guaranteed if $(S_k)$ is chosen uniformly positive definite. 
			If $\CS$ is strongly convex, this holds for $S_k=\nabla^2\CS(x_k)$, but 
			more sophisticated choices may be available for the problem at hand. 
			If $\CS$ is convex, we can use $S_k=\nabla^2\CS(x_k)+w I$ with some small $w>0$.
			In the unstructured setting $\CS\equiv 0$ this reads $S_k=w I$. 
			Note that in all these considerations the data-fitting term $\CD$ in \cref{eq_SO} can be non-convex.
		\end{remark}
		
		\subsubsection{Linear convergence under a \KL-type inequality}
		
		In this subsection we show the linear convergence of Algorithm~\ref{alg_hybrid} based on a \KL-type inequality.
		To this end, we briefly review two existing variants of the \KL~inequality and relate them to the 
		one that we use. 		
		
		\KL-type inequalities have recently been used in the convergence analysis of many optimization methods, cf. e.g. \cite{AMA05,BDLM10,ABRS10,FGP15,KNS16,LP18,BCN19,BDL22,LMQ23}. 
		They are most often applied in nonsmooth settings, but for our purposes the smooth case suffices. 
		These inequalities exist in different forms, e.g., local and global versions. 
		An example for a global variant is contained in \cite[Appendix~G]{KNS20}. 
		There, the authors say that the \emph{\KL~inequality with exponent $\frac12$} holds iff
		there exists $\mu>0$ such that 
		\begin{equation}\label{eq_KLineqlit}
			\CJ(x)-\CJopt\leq \frac{1}{\mu}\norm{\nabla\CJ(x)}^2 \qquad\forall x\in\CX
		\end{equation}
		is satisfied, where $\CJopt:=\inf_{x\in\CX}\CJ(x)$ is assumed finite. 
		However, other authors refer to \cref{eq_KLineqlit} as the \emph{\Lo~inequality}, cf. \cite[Introduction]{BDLM10}. 
		The setting with exponent $\frac12$ is typically used to derive linear convergence, cf. e.g. \cite{KNS20}. 
		Since we work with exponent $\frac12$ only, we do not discuss other exponents here. 
		Observe that \cref{eq_KLineqlit} implies that all stationary points are global minimizers, but that it does not imply convexity. 
		Since any constant function satisfies \cref{eq_KLineqlit}, it is clear that this inequality allows 
		for non-isolated minimizers. A more sophisticated example is given by functions of the form $x\mapsto g(Ax)$ with arbitrary $A\in\R^{m\times n}$ and strongly convex $g\in C^{1,1}$, 
		cf. \cite[Appendix~2.3]{KNS16}. While these examples are all convex, 
		not every convex and smooth function satisfies a \KL-type inequality, cf. \cite[Theorem~36]{BDLM10}.
		
		It is more common to speak of a \KL~inequality if there are $\epsilon,\mu>0$ such that  
		\begin{equation}\label{eq_KLineqlit2}
			\CJ(x)-\CJ(\xopt) \leq \frac{1}{\mu}\norm{\nabla\CJ(x)}^2 \qquad \forall x\in\{x:\CJ(\xopt)<\CJ(x)<\CJ(\xopt)+\epsilon\}
		\end{equation}
		is satisfied for a stationary point $\xopt$, cf. \cite[(1) with $\varphi(t)=t^\frac12$]{BDLM10}. 
		The inequality that we use is similar to \cref{eq_KLineqlit2}, but less stringent. 
		Specifically, we consider the sequence $(x_k)$ generated by Algorithm~\ref{alg_hybrid} and we recall from \cref{lem_welldef}
		that $(\CJ(x_k))$ is strictly monotonically decreasing and that $\CJopt:=\lim_{k\to\infty}\CJ(x_k)$ exists. 
		We demand that there are $\bar k,\mu>0$ such that 
		\begin{equation}\label{eq_PLcond}
			\CJ(x_k)-\CJopt \leq \frac{1}{\mu}\norm{\nabla\CJ(x_k)}^2 \qquad \forall k\geq\bar k. 
		\end{equation}
		We have not seen this \emph{algorithmic version} of the \KL~inequality in the literature. 
		It is not difficult to check that well-known \emph{error bound conditions} like the one in 
		\cite[Assumption~2]{TY09} imply \cref{eq_PLcond}. Thus, the following result holds in particular under any of those error bound conditions. 
		The significance of all aforementioned conditions including \cref{eq_PLcond} 
		is that they allow for minimizers that are neither locally unique nor have a regular Hessian. 
		
		As the second main result of this work we now show linear convergence of Algorithm~\ref{alg_hybrid} under \cref{eq_PLcond}.	
		We recall that the parameter $\sigma$ appears in the Armijo~condition~\cref{eq_armijocond}.
				
		\begin{theorem}\label{thm_rateofconvPL}
			Let \cref{ass_linconv} and \cref{eq_PLcond} hold.	
			Then there exists $\xopt$ such that 
			\begin{enumerate}
				\item[1)] there hold $\nabla\CJ(\xopt)=0$ and $\CJopt=\CJ(\xopt)$;
				\item[2)] the iterates $(x^k)$ converge r-linearly to $\xopt$;
				\item[3)] the gradients $(\nabla\CJ(x_k))$ converge r-linearly to zero;
				\item[4)] the function values $(\CJ(x_k))$ converge q-linearly to $\CJ(\xopt)$. 
				Specifically, we have 
				\begin{equation}\label{eq_finalineqinproof}
					\CJ(x_{k+1})-\CJ(\xopt)
					\leq \left(1-\frac{\sigma\alpha_k\mu}{\norm{B_k}}\right)\Bigl[\CJ(x_k)-\CJ(\xopt)\Bigr] \qquad \forall k\geq\bar k.
				\end{equation}
				The supremum of the term in round brackets is strictly smaller than 1. 
			\end{enumerate} 
		\end{theorem}	
			
		\begin{proof}
		The proof is divided into two parts. First we show 4) with $\CJ(\xopt)$ replaced by $\CJopt$. Then we prove 1)--3). 
		Continuity then implies that 4) holds with $\CJopt$ replaced by $\CJ(\xopt)$.\\			
		\textbf{Part 1: Q-linear convergence of $(\CJ(x_k))$}\\
		Setting $m_k:=\norm{B_k}^{-1}>0$ and using the Armijo condition \cref{eq_armijocond} we find for all $k\geq\bar k$  
		\begin{equation*}
			\begin{split}
				\CJ(x_{k+1})-\CJopt
				& = \CJ(x_{k+1})-\CJ(x_k)+\CJ(x_k)-\CJopt\\
				&\leq \sigma\alpha_k\nabla\CJ(x_k)^T d_k + \CJ(x_k)-\CJopt
				\leq -\sigma\alpha_k m_k\norm{\nabla\CJ(x_k)}^2 + \CJ(x_k)-\CJopt\\
				&\stackrel{\cref{eq_PLcond}}{\leq} -\sigma\alpha_k m_k\mu\bigl[\CJ(x_k)-\CJopt\bigr] + \CJ(x_k)-\CJopt
				= \left(1-\sigma\alpha_k m_k\mu\right) \bigl[\CJ(x_k)-\CJopt\bigr].
			\end{split}
		\end{equation*}
		This shows \cref{eq_finalineqinproof}. The boundedness of $(\norm{B_k})$ implies $\inf_k m_k>0$, so we obtain that the supremum of the term in round brackets is strictly smaller than one if there is $\alpha>0$ such that $\alpha_k\geq\alpha$ for all $k$. 
		The existence of such an $\alpha$ is well-known for the Wolfe--Powell conditions 
		since $(\norm{B_k^{-1}})$ is bounded and $\nabla\CJ$ is Lipschitz continuous in $\Omega$.
		For Armijo with backtracking the existence of $\alpha$ can be established similar as in the proof of \cref{thm_globconv}, starting at \cref{eq_yal} and using Case~1.\\			
		\textbf{Part 2: R-linear convergence of $(x_k)$ and $(\nabla \CJ(x_k))$}
		\\First we show that $\sup_k \alpha_k<\infty$.
		For all $k\geq\bar k$ the Armijo condition and \cref{eq_PLcond} imply that
		\begin{equation*}
			\alpha_k\sigma \nabla\CJ(x_k)^T B_k^{-1} \nabla\CJ(x_k)
			\leq \CJ(x_k)-\CJ(x_{k+1})
			\leq \CJ(x_k)-\CJopt
			\leq \frac{1}{\mu}\norm{\nabla\CJ(x_k)}^2,
		\end{equation*}	
		where we used that $\CJ(x_{k+1})\geq\CJopt$ since $(\CJ(x_k))$ is monotonically decreasing by \cref{lem_welldef}. 
		This yields for all $k\geq\bar k$ that $\alpha_k \leq \norm{B_k}/(\sigma\mu)$,
		proving $\sup_k\alpha_k<\infty$ as $(\norm{B_k})$ is bounded by assumption. 
		Next we deduce from $-\nabla\CJ(x_k)^T d_k = d_k^T B_k d_k$ and the Armijo condition that 
		\begin{equation*}
			\begin{split}
				\sum_{k=K}^\infty\frac{1}{\alpha_k}\frac{\norm{s_k}^2}{\norm{B_k^{-1}}}
				\leq\sum_{k=K}^\infty\alpha_k\frac{\norm{d_k}^2}{\norm{B_k^{-1}}}
				&\leq -\sum_{k=K}^{\infty}\alpha_k\nabla\CJ(x_k)^T d_k\\
				&\leq \frac{1}{\sigma}\sum_{k=K}^\infty \left[\CJ(x_{k})-\CJ(x_{k+1})\right]
				= \frac{1}{\sigma}\left[\CJ(x_K)-\CJopt\right]
			\end{split}
		\end{equation*}
		for any $K\in\N_0$. As $\beta:=\sup_{k\geq\bar k}\alpha_k\norm{B_k^{-1}}<\infty$, this yields
		for any $K\geq\bar k$
		\begin{equation*}
			\norm{s_K}^2 \leq \sum_{k=K}^\infty\norm{s_k}^2
			\leq \frac{\beta}{\sigma}\left[\CJ(x_K)-\CJopt\right],
		\end{equation*}
		hence $\norm{s_K} \leq \hat\beta\sqrt{\CJ(x_K)-\CJopt}$, where $\hat\beta:=\sqrt{\beta/\sigma}$.
		Since $(\CJ(x_k))$ converges q-linearly to $\CJopt$ by Part~1, this implies $\sum_k\norm{s_k}<\infty$, 
		so $(x_k)$ is convergent, i.e., there is $\xopt\in\CX$ with $\lim_{k\to\infty}x_k=\xopt$. 
		By continuity, $\CJ(\xopt)=\CJopt$. By \cref{thm_globconv} there holds
		$\nabla\CJ(\xopt)=0$. Furthermore, by applying 
		\cref{eq_finalineqinproof} to $\norm{s_K} \leq \hat\beta\sqrt{\CJ(x_K)-\CJopt}$ we infer 
		for all $K\geq\bar k$ 
		\begin{equation*}
			\Norm{x_K-\xopt}=\Norm{\sum_{k=K}^\infty s_k }
			\leq\sum_{k=K}^\infty\norm{s_k}\leq \frac{\hat\beta}{1-\nu}\sqrt{\CJ(x_K)-\CJ(\xopt)},
		\end{equation*} 
		where $\nu:=\sup_{k\geq\bar k} \sqrt{1-\sigma\alpha_k m_k\mu}$ satisfies $\nu<1$ by Part~1. 
		Thus, for all $K\geq\bar k$ and any $j\in\N_0$
		\begin{equation}\label{eq_KLi}
			\begin{split}
				\norm{x_{K+j}-\xopt}^2 
				&\leq \left(\frac{\hat\beta}{1-\nu}\right)^2\Bigl[\CJ(x_{K+j})-\CJ(\xopt)\Bigr]\\
				&\hspace*{-.1cm}\stackrel{\cref{eq_finalineqinproof}}{\leq} \left(\frac{\hat\beta\nu^j}{1-\nu}\right)^2 \Bigl[\CJ(x_K)-\CJ(\xopt)\Bigr]
				\leq \frac{L}{2} \left(\frac{\hat\beta\nu^j}{1-\nu}\right)^2 \norm{x_K-\xopt}^2
			\end{split}
		\end{equation}
		which demonstrates that $(x_k)$ converges r-linearly to $\xopt$. 
		The Lipschitz continuity of $\nabla\CJ$ together with \cref{eq_KLi} yields for any $K\geq\bar k$ and any $j\in\N_0$
		\begin{equation}\label{eq_KLi2}
			\begin{split}
				\norm{\nabla\CJ(x_{K+j})}^2 \leq L^2\norm{x_{K+j}-\xopt}^2 
				&\leq \left(\frac{L\hat\beta\nu^j}{1-\nu}\right)^2 \Bigl[\CJ(x_K)-\CJ(\xopt)\Bigr]\\
				&\leq \left(\frac{L\hat\beta \nu^j}{1-\nu}\right)^2\frac{1}{\mu}\norm{\nabla\CJ(x_K)}^2,
			\end{split}
		\end{equation}
		where the final inequality relies on \cref{eq_PLcond}. This proves the r-linear convergence of $(\nabla\CJ(x_k))$. 
	\end{proof}
	
	\begin{remark}\label{rem_boundednessBs}
		\phantom{induces linebreak}
		\begin{enumerate}
			\item[1)] \Cref{thm_rateofconvPL} does not claim that the stationary point $\xopt$ is a local minimizer. On the other hand, since $(\CJ(x_k))$ is
			strictly decreasing, $\xopt$ cannot be a local maximizer either.
			
			\item[2)] If $\CJ$ is strongly convex in the convex level set $\Omega$, then \cref{eq_PLcond} holds for $\bar k=0$.
			
			\item[3)] If \cref{eq_PLcond} holds for some $\bar k>0$, then it holds for all $\bar k\geq 0$ (after decreasing $\mu$ if need be). 
			
			\item[4)] The estimates \cref{eq_KLi} and \cref{eq_KLi2} quantify the r-linear convergence of $(x_k)$ and $(\nabla\CJ(x_k))$.
			
			\item[5)] As in \cref{rem_unifcontfindim} the statements concerning $\Omega_\delta$ in \cref{ass_globconv} and in \Cref{ass_linconv} 
			can be dropped if $\CX$ is finite dimensional and $(x_k)$ is bounded. 
						
		\end{enumerate}
	\end{remark}	
	
	\subsubsection{Linear convergence under local strong convexity}	
		
	The second result on the convergence rate of Algorithm~\ref{alg_hybrid} relies on the following lemma.
	
	\begin{lemma}\label{lem_aux}
		Let \cref{ass_linconv} hold except for the statements concerning $\Omega_\delta$. 
		Let $(x_k)$ have a cluster point $\xopt$ such that $\CJ\vert_{\CN}$ is strongly convex, where $\CN\subset\Omega$ is a convex neighborhood of $\xopt$.
		Then $\lim_{k\to\infty}x_k = \xopt$. 
	\end{lemma}	
	
	\begin{proof}
		From \cref{thm_globconv} we obtain $\CJ(\xopt)=\lim_{k\to\infty}\CJ(x_k)$ and $\nabla\CJ(\xopt)=0$ for the cluster point $\xopt$. 
		The strong convexity of $\CJ$ implies the growth condition
		\begin{equation}\label{eq_strconvgrc}
			\CJ(\xopt) + \kappa\norm{x-\xopt}^2 \leq\CJ(x)
		\end{equation}
		for all $x$ in $\CN$ and a constant $\kappa>0$. 
		Also due to the strong convexity, the \KL-inequality $\CJ(x)-\CJ(\xopt)\leq \mu^{-1} \norm{\nabla\CJ(x)}^2$ holds for all $x\in\CN$ and a constant $\mu>0$. 
		Let $\epsilon'\in(0,\delta/2]$. We have to show that there is $\bar k\geq 0$ such that $\norm{x_k-\xopt}\leq\epsilon'$ for all $k\geq \bar k$. 
		Owing to \cref{ass_globconv} the Armijo condition is satisfied for all $k$. 
		Hence, we have for all $x_k\in\CN$ 
		\begin{equation*}
			\alpha_k\sigma \nabla\CJ(x_k)^T B_k^{-1} \nabla\CJ(x_k)
			\leq \CJ(x_k)-\CJ(x_{k+1})
			\leq \CJ(x_k)-\CJ(\xopt) 
			\leq \frac{1}{\mu}\norm{\nabla\CJ(x_k)}^2,
		\end{equation*}	
		where we used that $\CJ(x_{k+1})\geq\CJ(\xopt)$ since $(\CJ(x_k))$ is monotonically decreasing. 
		Thus, for all $x_k\in\CN$ there holds
		\begin{equation*}
			\alpha_k \leq \frac{\norm{B_k}}{\sigma\mu}.
		\end{equation*}	
		Since $(\norm{B_k})$ is bounded by assumption, there holds $\norm{s_k}=\alpha_k\norm{d_k}
		\leq C\norm{d_k}$ for all $x_k\in\CN$ and a constant $C>0$. 
		As $(\norm{B_k^{-1}})$ is bounded by assumption, we infer 
		$\lim_{k\to\infty}\norm{d_k}=0$ from $\norm{d_k}\leq \norm{B_k^{-1}}\norm{\nabla\CJ(x_k)}$. In turn, there is $k_1>0$ such that 
		$\norm{s_k}\leq \epsilon'$ for all $k\geq k_1$ satisfying $x_k\in\CN$.
		Due to \cref{eq_strconvgrc} we have $\CJ(x)-\CJ(\xopt)\geq \kappa\epsilon'$ for all $x\in\CN\setminus\Ballop{\epsilon'}(\xopt)$.
		Letting $k_2\geq k_1$ be such that $\CJ(x_k)-\CJ(\xopt)<\kappa\epsilon'$ for all $k\geq k_2$, we deduce $x_k\not\in\CN\setminus\Ballop{\epsilon'}(\xopt)$ 
		for all $k\geq k_2$. Selecting $\bar k\geq k_2$ such that $x_{\bar k}\in\Ballop{\epsilon'}(\xopt)$, induction yields $x_k\in\Ballop{\epsilon'}(\xopt)$ for all $k\geq \bar k$.
	\end{proof}

	We now show linear convergence under a different set of assumptions than in \cref{thm_rateofconvPL}. 
	In contrast to \cref{thm_rateofconvPL}, the assumptions here ensure that $\xopt$ is a local minimizer. 
	For the special case $\CS\equiv 0$ and $S_k=0$ for all $k$, the following result may be viewed as an improved version of the classical convergence result \cite[Thm.~7.1]{Liu1989} from Liu and Nocedal on \LBFGS, the most notable improvement being that strong convexity is required only locally. 
	Our result covers Armijo step sizes, whereas \cite[Thm.~7.1]{Liu1989} does not. Also, our differentiability requirements are lower and we prove a convergence rate for $(\nabla\CJ(x_k))$, which is not done in \cite[Thm.~7.1]{Liu1989}.
	
	\begin{theorem}\label{thm_rateofconvstrongminimizer}
	Let \cref{ass_linconv} hold except for the statements concerning $\Omega_\delta$.
	Let $(x_k)$ have a cluster point $\xopt$ such that $\CJ\vert_{\CN}$ is $\kappa$-strongly convex, where $\CN\subset\Omega$ is a convex neighborhood of $\xopt$.
	Then 
	\begin{enumerate}
		\item[1)] there holds $\CJ(\xopt) + \kappa\norm{x-\xopt}^2 \leq\CJ(x)$ for all $x\in\CN$;
		\item[2)] the iterates $(x_k)$ converge r-linearly to $\xopt$;
		\item[3)] the gradients $(\nabla\CJ(x_k))$ converge r-linearly to zero;
		\item[4)] the function values $(\CJ(x_k))$ converge q-linearly to $\CJ(\xopt)$. 
		Specifically, if $\bar k$ is such that $x_k\in\CN$ for all $k\geq\bar k$, then we have
		\begin{equation*}
			\CJ(x_{k+1})-\CJ(\xopt)
			\leq \left(1-\frac{2\sigma\alpha_k \kappa}{\norm{B_k}}\right)\Bigl[\CJ(x_k)-\CJ(\xopt)\Bigr] \qquad \forall k\geq\bar k.
		\end{equation*}
		The supremum of the term in round brackets is strictly smaller than 1. 
	\end{enumerate}
	\end{theorem}	
	
	\begin{proof}
		By \cref{lem_aux} there is $\bar k$ such that $x_k\in\CN$ for all $k\geq\bar k$.
		Moreover, by \cref{thm_globconv} there holds $\nabla\CJ(\xopt)=0$,
		so the strong convexity of $\CJ$ implies that \cref{eq_PLcond} is satisfied for all $k\geq\bar k$ with constant $\mu:=2\kappa$. 
		Parts 2)--4) now follow from \cref{thm_rateofconvPL}. 
		It is well-known that the strong convexity of $\CJ$ implies 1).
	\end{proof}
	
	\begin{remark}
		Part~1) of \cref{thm_rateofconvstrongminimizer} shows, in particular, that $\xopt$ is a strict local minimizer. 
		Similar comments as in \cref{rem_boundednessBs}~2)--4) also apply to \cref{thm_rateofconvstrongminimizer}.
		The \emph{existence} of cluster points is for instance ensured if $(x_k)$ is bounded and $\CX$ is finite dimensional.
	\end{remark}
					

\section{Numerical experiments}\label{sec:experiments}
	
	We study the performance of \ref{alg_hybrid} for two classes of inverse problems: 
	First for \emph{non-convex image registration problems}.
	These are challenging real-world problems that demonstrate the effectiveness of the	algorithm. 
	Second for \emph{strictly convex quadratics}. 
	These are constructed problems that allow us to examine the convergence properties of the algorithm more closely. 
	
	Throughout, we consider seven choices of the seed matrix.  
	Two are based on the classical \LBFGS~method Algorithm~\ref{algo:lbfgs} with $H_k^{(0)}=\hat\tau_k I$, 
	where either $\hat\tau_k=\hat\tau_k^\Hy$ for all $k$ or $\hat\tau_k=\hat\tau_k^\Hs$; cf.~\cref{def:tau:H}. 
	These two algorithms are state-of-the-art in image registration. They do not require to solve linear systems, which distinguishes them from 
	the other five choices. 
	The five other choices are based on the structured approach Algorithm~\ref{alg_hybrid} and 
	they use $B_k^{(0)}=\tau_k I + \nabla^2\CS(x_k)$, 
	where $\tau_k$ either takes one of the four values defined in \cref{def:tau:B} or is computed by Algorithm~\ref{alg_adaptivechoiceoftau}. 
	We prefix the schemes by $H$ or $B$ when we refer to them in this section; e.g., we write $\mathrm{B\Bp}$ for the choice $\tau^\Bp$ and $\mathrm{H\Hy}$ for $\hat\tau^\Hy$.
	
	\renewcommand{\Hy}{{\mathrm{Hy}}}
	\renewcommand{\Hs}{{\mathrm{Hs}}}
	\renewcommand{\Hp}{{\mathrm{Hs}}}
	\renewcommand{\Bp}{{\mathrm{Bs}}}
	\renewcommand{\Bz}{{\mathrm{Bz}}}
	\renewcommand{\Bu}{{\mathrm{Bu}}}
	\renewcommand{\GM}{{\mathrm{Bg}}}
	
	\subsection{Image registration}\label{sec:IR}
		
		\paragraph{Basics}
		We solve 22 real-life large-scale problems from medical image registration (IR) \cite{FAIR09}. 
		Registration problems are generally highly non-convex and ill-posed. Given a
		pair of images $T$ and $R$, the goal is to find a transformation field $\phi$
		such that the transformed image $T(\phi)$ is similar to $R$, i.e., 
		$T\circ\phi\approx R$. To determine $\phi$, we solve an unconstrained optimization problem
		\begin{equation*}
			\min_\phi \CJ(\phi) = \CD(\phi;T,R)+ \alpha S(\phi),
		\end{equation*}
		where $\CD$ measures the similarity between the transformed image $T(\phi)$ and
		$R$. Typical choices for similarity measures are the sum of squared
		difference (SSD) $\CD=\|T\circ\phi-R\|_{L_2}^2$, mutual information (MI) \cite{CMDFVSM95,C98,VW95,V95}, and normalized gradient fields (NGF) \cite{Haber2006}.
		The regularizer $\CS=\alpha S$ guarantees that the problem is solvable and it enforces
		smoothness in the field. Commonly used regularizers are 
		elastic~\cite{Broit1981}, given by $\CS=\int\mu\langle\nabla\phi,\nabla\phi\rangle
		+(\lambda+\mu)(\mathrm{div}\cdot\phi)^2$, curvature~\cite{Fischer2004}, and the non-quadratic hyperelasticity \cite{BMR13}. 
		We follow the ``discretize-then optimize'' approach. 
		To obtain a clearer comparison we work with a fixed discretization instead of a multilevel approach.
		In consequence, for each problem the transformation field satisfies $\phi \in \R^n$ with a fixed $n$ that defines the number of variables. 
		The values of $n$ are provided in \cref{tab:IRproblems}.
			
		\paragraph{Problems under consideration}
		The 22 test cases, listed in \cref{tab:IRproblems}, cover many different registration models. The 
		three-dimensional (3D) lung CT images are from the well-known DIR dataset \cite{Knig2018,Castillo2009}, 
		and the rest of the datasets are from \cite{FAIR09}.
		The test cases comprise the fidelity measures SSD, NGF and MI. They include 
		a quadratic first order (Elas), a quadratic second order (Curv) and a non-quadratic first order (H-elas) regularizer.
		Computing the full Hessian of the hyperelastic regularizer is expensive; 
		as in \cite{BMR13} we replace it by a Gauss-Newton-like approximation. 
		Recall that the convergence analysis allows for $S_k\neq\nabla^2\CS(x_k)$.
				
\begin{table}[ht]
\footnotesize
\caption{Non-quadratic test cases (TC) from image registration problems. 
Data-fidelity (MI, NGF, SSD) and regularization (Curvature, Elasticity, Hyperelasticity) are
denoted by $\CD$ and $\CS$, respectively; the regularization parameter is
$\alpha$. The problem size is $n=d\cdot\prod_{k=1}^dm_k$, where $d\in\{2,3\}$ denotes
the data dimensionality (2D or 3D) and $m_k$ the corresponding data resolution. 
The data resolution for Hand, PET-CT and MRI data are $128 \times 128$, for Disc data $16\times16$, Lung data $64 \times 64 \times X$ where $X \in [24,28]$, and Brain data $32 \times 16 \times 32$.
The last column reports the initial target registration error (Initial TRE), cf.~\cref{sec:IR}.
}\label{tab:IRproblems}

\centering
\begin{tabular}{rlrllrc}
	\toprule
	TC
	& Dataset & $n$
	& $\CD$ & $\CS$ & $\alpha$
	& Initial TRE
	\\
	\midrule
	1 & 2D-Hands & $32\,768$	 & SSD & Curv   & $1.5 \cdot 10^3$ & 1.04 (0.62)\\
	2 & 2D-Hands & $32\,768$	 & SSD & Elas   & $1.5 \cdot 10^3$ & 1.04 (0.62)\\
	3 & 2D-Hands & $32\,768$	 & NGF  & Curv   & $10^{-2}$ & 1.04 (0.62)\\
	4 & 2D-Hands & $32\,768$	 & NGF  & Elas   & $1$  & 1.04 (0.62)\\
	5 & 2D-Hands & $32\,768$	 & MI  & Curv   & $5 \cdot 10^{-3}$ & 1.04 (0.62)\\
	6 & 2D-Hands & $32\,768$	 & MI  & Elas   & $5 \cdot 10^{-3}$  & 1.04 (0.62)\\
	7 & 2D-PET-CT & $32\,768$	 & MI  & Elas   & $10^{-4}$ & N.A.\\
	8 & 2D-PET-CT & $32\,768$	 & MI  & Curv   & $10^{-1}$ & N.A.\\
	9 & 2D-PET-CT & $32\,768$	 & NGF  & Elas   & $5 \cdot 10^{-2}$ & N.A.\\
	10 & 2D-PET-CT & $32\,768$	 & NGF  & Curv   & $10$ & N.A.\\
    11 & 2D-MRI-head & $32\,768$	 & MI  & Elas   & $10^{-3}$ & N.A.\\
	12 & 2D-MRI-head & $32\,768$	 & NGF  & Elas   & $10^{-1}$ & N.A.\\	
	13 & 3D-Lung  & $294\,912$ & NGF & Curv   & $10^2$ & 3.89 (2.78) \\
    14 & 3D-Lung  & $344\,064$ & NGF & Curv   & $10^2$ & 4.34 (3.90)\\
    15 & 3D-Lung  & $319\,488$ & NGF & Curv   & $10^2$ & 6.94 (4.05) \\
    16 & 3D-Lung  & $307\,200$ & NGF & Curv   & $10^2$ & 9.83 (4.86) \\
    17 & 3D-Lung  & $331\,776$ & NGF & Curv   & $10^2$ & 7.48 (5.51) \\
    18 & 3D-Brain  & $49\,152$	 & SSD & H-elas & $(100,10,100)$ & N.A. \\
    19 & 2D-Hands &	$32\,768$ & SSD & H-elas & $(10^3,20)$ & 1.04 (0.62)\\
    20 & 2D-Hands &	$32\,768$ & NGF & H-elas & $(1,1)$ & 1.04 (0.62)\\
    21 & 2D-Hands &	$32\,768$ & MI & H-elas & $(10^{-3},1)$ & 1.04 (0.62)\\
    22 & 2D-Disc  & $512$	 & SSD & H-elas & $(100,20)$ & N.A. \\
	\bottomrule    
\end{tabular}
\end{table}

		\paragraph{Evaluation}
		We use run-time and solution accuracy as the main criteria to evaluate the performance of the optimization methods.
		Let us, however, stress that for image registration problems the quantification of the quality of results is still an open question.  
		In particular, the value of the objective $\CJ$ is typically \emph{not} a good measure of quality. 
		Since ground truth transformation fields are not available in practice, a common approach is based on the \emph{target registration error} 
		(TRE) \cite{Fitzpatrick2001}. The idea is to identify landmarks in the datasets
		$R$, $T$ and $T\circ\phi$, and to compute the mean error and standard deviation
		of the landmark locations in the Euclidean distance. We employ user supplied landmarks for the 14 out of 22 problems where they are available. For the 8 problems marked ``N.A.'' in \cref{tab:IRproblems}, there are no landmarks available.
				
		We use the well-established \emph{performance profiles} of Dolan and Moré \cite{Dolan2002} to visualize the performance of optimization methods. Suppose we want to compare the performance of $n_s$ optimization methods on $n_p$ problems. Denote by $S$ the set of methods and by $P$ the set of problems. The performance profile for method $s\in S$ is the function $\rho_s:[1,\infty)\rightarrow [0,1]$ given by 
		\begin{equation*}
			\rho_{s}(\tau) = \frac{\left\lvert\bigl\{p\in P: \, r_{p,s} \leq \tau\bigr\}\right\rvert}{n_p}, \qquad \text{where } r_{p,s} = \frac{t_{p,s}}{\min\bigl\{t_{p,\sigma}: \, \sigma \in S\bigr\}}.
		\end{equation*}
		Here, $t_{p,\sigma}$ denotes a performance metric (e.g., run-time) of method $\sigma$ on problem $p$. We see that $\rho_s$ is the cumulative distribution function wrt. the performance metric $t$. 
		Note that $\tau$ in $\rho_{s}(\tau)$ is not related to the scaling factor $\tau$ used for seed matrices, but both are standard notation. 
		
		\paragraph{Parameter values and termination criteria}
		The image processing operations are carried out matrix free with the open-source image registration toolbox \FAIRTEXT~\cite{FAIR09}. 
		The initial guess $\phi_0$ in all experiments corresponds to a discretization of the identity map $\phi(x)=x$. 
		The regularization parameter $\alpha$ is chosen in such a way that it yields the lowest TRE without introducing unphysical foldings in
		the transformation field. For the stopping criteria of the optimization methods we follow \cite[p.~78]{FAIR09}. 
		That is, we stop if all of the conditions
		\begin{itemize}
			\item $\lvert\CJ(x_k) - \CJ(x_{k-1})\rvert \le 10^{-5}(1 + \lvert\CJ(x_0)\rvert)$, 
			\item $\|x_k - x_{k-1}\|\le 10^{-3}(1 + \|x_k\|)$,
			\item $\|\nabla\CJ(x_k)\| \leq 10^{-3}(1 + \lvert\CJ(x_0)\rvert)$
		\end{itemize}
		are satisfied. 
		We use $\ell = 5$ in \ref{alg_hybrid} unless stated otherwise. 
		The remaining parameter values of Algorithm~\ref{alg_hybrid} and Algorithm~\ref{alg_adaptivechoiceoftau} are specified in \cref{tab_paramvalues}.
				
		\begin{table}
			\footnotesize
			\caption{Parameter values for Algorithms~\ref{alg_hybrid} and \ref{alg_adaptivechoiceoftau}}
			\label{tab_paramvalues}
			\centering
		\begin{tabular}{c|c|c|c|c|c|c|c|c|c|c|c|c}
			\toprule
			\multicolumn{5}{c|}{Algorithm~\ref{alg_hybrid}} & \multicolumn{8}{c}{Algorithm~\ref{alg_adaptivechoiceoftau}} \\
			\midrule
			$c_s$ & $c_0$ & $C_0$ & $c_1$ & $c_2$ & $\Delta_0$ & $\Delta_1$ & $\eps_0$ & $\eps_1$ & $\eta_0$ & $\eta_1$ & $\eta_2$ & $\beta$\\
			\midrule
			$10^{-9}$ & $10^{-6}$ & $10^6$ & $10^{-6}$ & $1$ & $0.75$ & $0.1$ & $10^{-3}$ & $10^{-4}$ & $0.025$ & $0.1$ & $0.05$ & $0.01$\\
			\bottomrule
		\end{tabular}
		\end{table}
		
		To satisfy the Wolfe--Powell conditions we use the Moré--Thuente \cite{MT94} line search routine from the Poblano toolbox \cite{Poblano,Pob2} with the parameter values from \cref{tab_paramvalues3}. 
		These are the same values that \cite{BDLP22} use for their structured \LBFGS~methods to which we compare our method below. 
		We use the Armijo routine from FAIR \cite{FAIR09} with the values specified in \cref{tab_paramvalues3}. 
										
		\begin{table}
			\footnotesize
			\caption{Parameter values for the Moré--Thuente, respectively, the Armijo line search}
			\label{tab_paramvalues3}
			\centering
		\begin{tabular}{c|c|c|c|c|c|c|c}
			\toprule
			\multicolumn{6}{c|}{Moré--Thuente (Poblano)} & \multicolumn{2}{c}{Armijo (FAIR)} \\
			\midrule
			$ftol$ & $gtol$ & $maxfev$ & $stpmax$ & $stpmin$ & $xtol$ & $LSmaxIter$ & $LSreduction$\\ 
			\midrule
			$10^{-4}$ & $0.9$ & $3000$ & $2$ & $0$ & $10^{-6}$ & $50$ & $10^{-4}$\\
			\bottomrule
		\end{tabular}
		\end{table}									
		
		\paragraph{Experimental comparison of different line searches}
		
		The step sizes for \LBFGS~are most often computed to satisfy the Wolfe--Powell conditions,
		but for the image registration problems we prefer to satisfy only the Armijo condition. 
		The reason for this is that the additional gradient evaluations that the Wolfe--Powell conditions require are quite expensive in image registration. 
		To substantiate this statement we compare Wolfe to Armijo on the 22~image registration problems, where each problem is solved with the seven aforementioned choices for the seed matrix. 
		The performance profiles based on these $22 \cdot 7=154$ optimization problems are provided in \cref{fig:perf_WA}.
		
		\pgfplotsset{myaxisstyle/.style={
		axis y line=left,
		scale = 0.55,
		title style={at={(0.8,1.7)}},
		axis lines = left,
		xlabel = $\tau$,
		ymin=0, ymax=1.1,
		ylabel style={at={(-0.3,1)}},
		xlabel style={at={(1,-0.3)}},
		minor tick num=2,
		grid=both,
		grid style={line width=.1pt, draw=gray!05},
		cycle list name=color list,
		line width=0.8pt,
		x tick label style={
			/pgf/number format/.cd,
			fixed,
			precision=3,
			/tikz/.cd
		},
}}

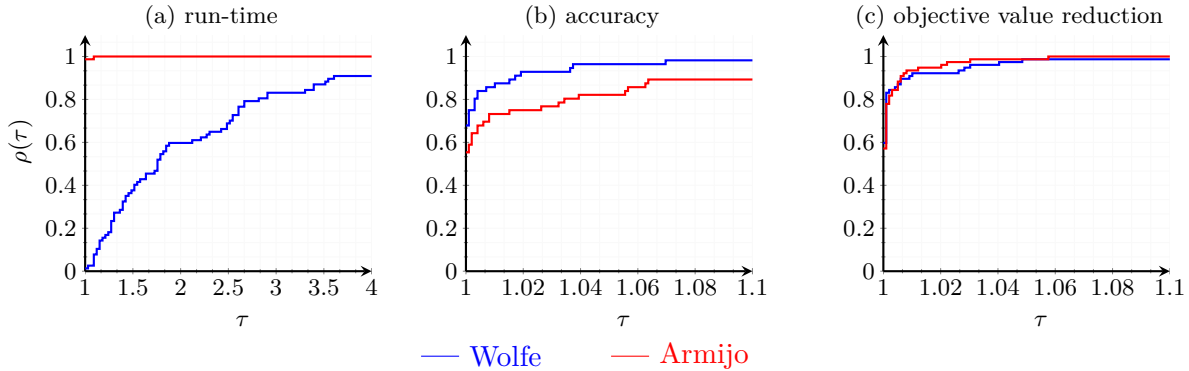
\begin{figure*}[ht!]
    \centering
    \begin{minipage}[b]{0.32\linewidth}
    \pgfplotstableread[col sep=comma]{data/WA_time.csv}\loadedtable
    \begin{tikzpicture}[font=\footnotesize,inner sep=2pt, outer sep=0pt]
      \begin{axis}[
        myaxisstyle,  
        title = {(a) run-time},
        ylabel = $\rho(\tau)$,
        xmin=1, xmax=4,
      ]
      \addplot[blue] table[x=x, y=c1]{\loadedtable};
      \addplot[red] table[x=x, y=c2]{\loadedtable};
      \end{axis}
    \end{tikzpicture}
    \end{minipage}
    \hfill
    \begin{minipage}[b]{0.32\linewidth}
   		\pgfplotstableread[col sep=comma]{data/WA_mTRE.csv}\loadedtable
    	\begin{tikzpicture}[font=\footnotesize,inner sep=2pt, outer sep=0pt]
    		\begin{axis}[
    			myaxisstyle,  
    			title = {(b) accuracy},
    			xmin=1, xmax=1.1,
    			]
    			\addplot[blue] table[x=x, y=c1]{\loadedtable};
    			\addplot[red] table[x=x, y=c2]{\loadedtable};
    		\end{axis}
    	\end{tikzpicture}
    \end{minipage}
    \hfill
    \begin{minipage}[b]{0.32\linewidth}
	\pgfplotstableread[col sep=comma]{data/WA_red.csv}\loadedtable
	\begin{tikzpicture}[font=\footnotesize,inner sep=2pt, outer sep=0pt]
	\begin{axis}[
		myaxisstyle,  
		title = {(c) objective value reduction},
		xmin=1, xmax=1.1,
		]
		\addplot[blue] table[x=x, y=c1]{\loadedtable};
		\addplot[red] table[x=x, y=c2]{\loadedtable};
	\end{axis}
	\end{tikzpicture}
    \end{minipage}
     \begin{minipage}[b]{0.32\linewidth}
    	\begin{tikzpicture}
    	\draw[blue,solid] (-1,0) -- (-0.5,0) node[right] {Wolfe} ;
    	\draw[red,solid] (1.5,0) -- (2,0) node[right] {Armijo} ;
    	\end{tikzpicture}
    \end{minipage}
    \caption{Performance profiles comparing Wolfe and Armijo line-search algorithms on IR problems;
    Armijo yields a significantly lower run-time at the expense of a somewhat lower accuracy.
    }
    \label{fig:perf_WA}
\end{figure*}

		\Cref{fig:perf_WA} reveals that Armijo is faster than Wolfe on almost all problems and often by a wide margin; 
		e.g., for around $40\%$ of the problems Wolfe requires twice as much or more run-time. 
		On the other hand, Wolfe is superior in terms of accuracy (=TRE). For instance, Armijo solves $89\%$ of the problems with an accuracy that deviates less than $10\%$ from the best obtained accuracy, whereas this value is $98\%$ of the problems for Wolfe. 
		Although in image registration we are mainly interested in run-time and accuracy, 
		we note that both line search strategies are similarly effective in decreasing the cost function. 
		In conclusion, Wolfe has an edge over Armijo in terms of accuracy, but Armijo is much faster. In the remaining experiments we use Armijo.
		We mention that the authors of \cite{ACG18} report a similar observation and also work with backtracking and Armijo.
				
		\paragraph{Accuracy of the linear solver}
		The linear system involved in the computation of $d_k$, cf. \cref{sec_tlrdisc}, 
		is solved inexactly and matrix free by a Jacobi preconditioned minimal residual method (PMINRES)~\cite{Paige1975}. 
		\Cref{fig:perf_minres} displays how the level of inexactness affects the performance of Algorithm~\ref{alg_hybrid}. 
		As it turns out, the smallest tolerance \emph{(maxiter,tol)} $=(500,10^{-6})$ yields the worst accuracy for most of the problems.
		Also, the largest tolerance $10^{-1}$ produces less accurate results than $10^{-2}$, albeit in less 
		run-time. 
		Henceforth, we use \emph{(maxiter,tol)} $=(50,10^{-2})$ which achieves the top accuracy.
		We stress that in our experiments, PMINRES outperformed the preconditioned conjugate gradients method. 
		Recent research suggests that this may be attributable to the early stopping, cf. \cite{FoSu12,YR21,RLXM22}.
		
		\pgfplotsset{myaxisstyle/.style={
		axis y line=left,
		scale = 0.55,
		title style={at={(0.8,1.7)}},
		axis lines = left,
		xlabel = $\tau$,
		ymin=0, ymax=1,
		ylabel style={at={(-0.3,1)}},
		xlabel style={at={(1,-0.3)}},
		minor tick num=2,
		grid=both,
		grid style={line width=.1pt, draw=gray!05},
		cycle list name=color list,
		line width=0.8pt,
		x tick label style={
			/pgf/number format/.cd,
			fixed,
			precision=3,
			/tikz/.cd
		},
}}

\begin{figure*}[ht!]
    \centering
    \begin{minipage}[b]{0.32\linewidth}
    \pgfplotstableread[col sep=comma]{data/mr_time.csv}\loadedtable
    \begin{tikzpicture}[font=\footnotesize,inner sep=2pt, outer sep=0pt]
      \begin{axis}[
        myaxisstyle,  
        title = {(a) run-time},
        ylabel = $\rho(\tau)$,
        xmin=1, xmax=2,
      ]
      \addplot[blue] table[x=x, y=c1]{\loadedtable};
      \addplot[red] table[x=x, y=c2]{\loadedtable};
      \addplot[orange] table[x=x, y=c3]{\loadedtable};
      \addplot[violet] table[x=x, y=c4]{\loadedtable};
      \addplot[teal] table[x=x, y=c5]{\loadedtable};
      \end{axis}
    \end{tikzpicture}
    \end{minipage}
    \begin{minipage}[b]{0.32\linewidth}
    	\pgfplotstableread[col sep=comma]{data/mr_mTRE.csv}\loadedtable
    	\begin{tikzpicture}[font=\footnotesize,inner sep=2pt, outer sep=0pt]
    	\begin{axis}[
    	myaxisstyle,  
    	title = {(b) accuracy},
    	xmin=1, xmax=1.05,
    	ymin = 0.2,
    	]
      \addplot[blue] table[x=x, y=c1]{\loadedtable};
		\addplot[red] table[x=x, y=c2]{\loadedtable};
		\addplot[orange] table[x=x, y=c3]{\loadedtable};
		\addplot[violet] table[x=x, y=c4]{\loadedtable};
		\addplot[teal] table[x=x, y=c5]{\loadedtable};
    	\end{axis}
    	\end{tikzpicture}
    \end{minipage}
	\begin{minipage}[b]{0.32\linewidth}
		\pgfplotstableread[col sep=comma]{data/mr_redt.csv}\loadedtable
		\begin{tikzpicture}[font=\footnotesize,inner sep=2pt, outer sep=0pt]
		\begin{axis}[
		myaxisstyle,  
		title = {(c) objective value reduction},
		xmin=1, xmax=1.02,
		ymin = 0.2,
		]
		\addplot[blue] table[x=x, y=c1]{\loadedtable};
		\addplot[red] table[x=x, y=c2]{\loadedtable};
		\addplot[orange] table[x=x, y=c3]{\loadedtable};
		\addplot[violet] table[x=x, y=c4]{\loadedtable};
		\addplot[teal] table[x=x, y=c5]{\loadedtable};
		\end{axis}
		\end{tikzpicture}
	\end{minipage}
     \begin{minipage}[b]{\linewidth}
    	\centering
    	\begin{tikzpicture}[font=\footnotesize,inner sep=1.5pt, outer sep=0pt]
    	\draw[blue,solid] (0.5,0) -- (1,0) node[right,align=left] {maxiter: $50$, \\ tol: $10^{-1}$} ;
    	\draw[red,solid] (3.5,0) -- (4,0)   node[right,align=left] {maxiter: $50$, \\ tol: $10^{-2}$} ;
    	\draw[orange,solid] (6.5,0) -- (7,0)  node[right,align=left] {maxiter: $100$, \\ tol: $10^{-1}$} ;
    	\draw[violet,solid] (9.5,0) -- (10,0)  node[right,align=left] {maxiter: $100$, \\ tol: $10^{-2}$} ;
    	\draw[teal,solid] (12.5,0) -- (13,0)  node[right,align=left] {maxiter: $500$, \\ tol: $10^{-6}$} ;
    	\end{tikzpicture}
    \end{minipage}
    \caption{Performance profiles for \ref{alg_hybrid}~with liner solver PMINRES. 
    	PMINRES terminates if the iteration count reaches \emph{maxiter} or the relative residual is smaller than \emph{tol}.
    	}
    \label{fig:perf_minres}
\end{figure*}
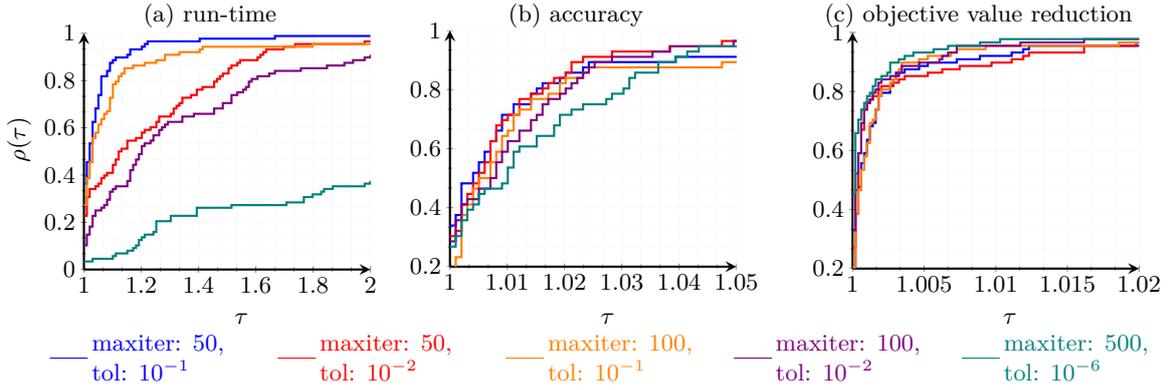	
		
		\paragraph{Experimental comparison of different seed matrices}
		Next we compare the seven choices of the seed matrix. 
		The results in \cref{fig:perf_HB} show that the structured choices outperform the unstructured ones. 
		Specifically, the two classical \LBFGS~variants require more run-time than Algorithm~\ref{alg_hybrid} with \ref{alg_adaptivechoiceoftau} while producing substantially less accurate solutions. This statement holds for all structured methods, although not depicted.
		We conclude that the structured methods are preferable for image registration. 
		
		\pgfplotsset{myaxisstyle/.style={
		axis y line=left,
		scale = 0.55,
	    title style={at={(0.8,1.7)}},
		axis lines = left,
		ylabel style={at={(-0.3,1)}},
		xlabel style={at={(1,-0.3)}},
		minor tick num=1,
		grid=both,
		grid style={line width=.1pt, draw=gray!05},
		cycle list name=color list,
		line width=0.8pt,
	    x tick label style={
			/pgf/number format/.cd,
			fixed,
			precision=3,
			/tikz/.cd
		}
}}

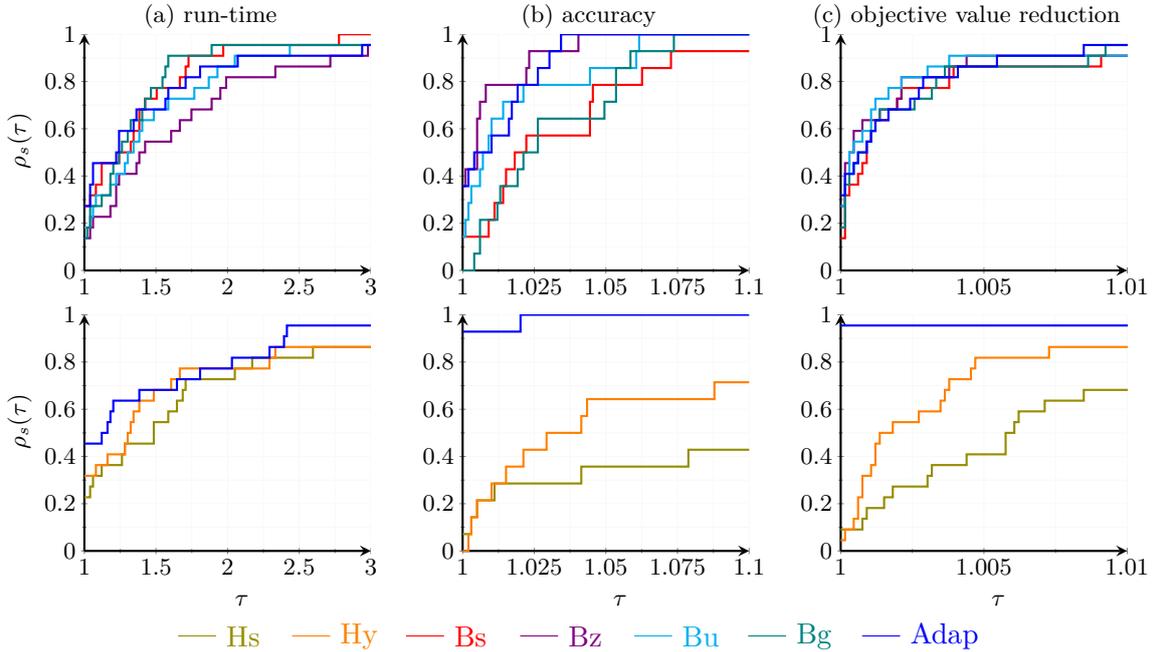
\begin{figure*}[ht!]
    \centering
    \begin{minipage}[b]{0.33\linewidth}
    \pgfplotstableread[col sep=comma]{data/B_time.csv}\loadedtable
    \begin{tikzpicture}[font=\footnotesize,inner sep=2pt, outer sep=0pt]
      \begin{axis}[
        myaxisstyle,  
        xtick={1,1.5,2,2.5,3},
        title = {(a) run-time},
        xmin=1, xmax=3, ymin=0, ymax=1,
        ylabel = $\rho_{s}(\tau)$,
      ]
      \addplot[red] table[x=x, y=c1]{\loadedtable};
      \addplot[violet] table[x=x, y=c2]{\loadedtable};
      \addplot[cyan] table[x=x, y=c3]{\loadedtable};
      \addplot[teal] table[x=x, y=c4]{\loadedtable};
      \addplot[blue] table[x=x, y=c5]{\loadedtable};
      \end{axis}
    \end{tikzpicture}
    \end{minipage}
    \begin{minipage}[b]{0.30\linewidth}
	  	\pgfplotstableread[col sep=comma]{data/B_mTRE.csv}\loadedtable
	\begin{tikzpicture}[font=\footnotesize,inner sep=2pt, outer sep=0pt]
		\begin{axis}[
			myaxisstyle,
			xtick = {1,1.025,1.05,1.075,1.1},
			title = {(b) accuracy},
			xmin=1, xmax=1.1, ymin=0, ymax=1,
			]
			\addplot[red] table[x=x, y=c1]{\loadedtable};
			\addplot[violet] table[x=x, y=c2]{\loadedtable};
			\addplot[cyan] table[x=x, y=c3]{\loadedtable};
			\addplot[teal] table[x=x, y=c4]{\loadedtable};
			\addplot[blue] table[x=x, y=c5]{\loadedtable};
		\end{axis}
	\end{tikzpicture}
    \end{minipage}
    \begin{minipage}[b]{0.30\linewidth}
    	\pgfplotstableread[col sep=comma]{data/B_red.csv}\loadedtable
    	\begin{tikzpicture}[font=\footnotesize,inner sep=2pt, outer sep=0pt]
    		\begin{axis}[
    			myaxisstyle, 
    			xtick = {1,1.005,1.01},
    			title = {(c) objective value reduction},
    			xmin=1, xmax=1.01, ymin=0, ymax=1,
    			]
    			\addplot[red] table[x=x, y=c1]{\loadedtable};
    			\addplot[violet] table[x=x, y=c2]{\loadedtable};
    			\addplot[cyan] table[x=x, y=c3]{\loadedtable};
    			\addplot[teal] table[x=x, y=c4]{\loadedtable};
    			\addplot[blue] table[x=x, y=c5]{\loadedtable};
    		\end{axis}
    	\end{tikzpicture}
    \end{minipage}
    \begin{minipage}[b]{0.33\linewidth}
   	\pgfplotstableread[col sep=comma]{data/HB_time.csv}\loadedtable
   	\begin{tikzpicture}[font=\footnotesize,inner sep=2pt, outer sep=0pt]
   	\begin{axis}[
	   	myaxisstyle,
	   	xtick={1,1.5,2,2.5,3},
	   	title = {},
	   	xmin=1, xmax=3, ymin=0, ymax=1,
	   	xlabel = $\tau$, ylabel = $\rho_{s}(\tau)$,
	   	]
			\addplot[olive] table[x=x, y=c1]{\loadedtable};
			\addplot[orange] table[x=x, y=c2]{\loadedtable};
			\addplot[blue] table[x=x, y=c3]{\loadedtable};
	   	\end{axis}
	   	\end{tikzpicture}
	\end{minipage}
    \begin{minipage}[b]{0.30\linewidth}
	   	\pgfplotstableread[col sep=comma]{data/HB_mTRE.csv}\loadedtable
	\begin{tikzpicture}[font=\footnotesize,inner sep=2pt, outer sep=0pt]
		\begin{axis}[
			myaxisstyle,
			xmin=1, xmax=1.1, ymin=0, ymax=1,
			xlabel = $\tau$, 
			xtick = {1,1.025,1.05,1.075,1.1},
			]
			\addplot[olive] table[x=x, y=c1]{\loadedtable};
			\addplot[orange] table[x=x, y=c2]{\loadedtable};
			\addplot[blue] table[x=x, y=c3]{\loadedtable};
		\end{axis}
	\end{tikzpicture}
   \end{minipage}
    \begin{minipage}[b]{0.30\linewidth}
	   	\pgfplotstableread[col sep=comma]{data/HB_red.csv}\loadedtable
	\begin{tikzpicture}[font=\footnotesize,inner sep=2pt, outer sep=0pt]
	\begin{axis}[
		myaxisstyle,
		xmin=1, xmax=1.01, ymin=0, ymax=1,
		xtick = {1,1.005,1.01},
		xlabel = $\tau$, 
		]
		\addplot[olive] table[x=x, y=c1]{\loadedtable};
		\addplot[orange] table[x=x, y=c2]{\loadedtable};
		\addplot[blue] table[x=x, y=c3]{\loadedtable};
	\end{axis}
	\end{tikzpicture}
   \end{minipage}
    \begin{minipage}[b]{\linewidth}
    	\centering
    	\begin{tikzpicture}
    	\draw[olive,solid] (0.5,0) -- (1.0,0) node[right] {$\Hp$} ;
    	\draw[orange] (2.0,0) -- (2.5,0) node[right] {$\Hy$};
    	\draw[red] (3.5,0) -- (4.0,0) node[right] {$\Bp$};
    	\draw[violet] (5,0) -- (5.5,0) node[right] {$\Bz$};
    	\draw[cyan] (6.5,0) -- (7,0) node[right] {$\Bu$} ;
    	\draw[teal] (8,0) -- (8.5,0) node[right] {$\GM$} ;
    	\draw[blue] (9.5,0) -- (10,0) node[right] {$\Adap$};
    	\end{tikzpicture}
    \end{minipage}
    
    \caption{Performance profiles for different choices of the seed matrix. 
    	The first row compares $\Bp$, $\Bz$, $\Bu$, $\GM$ and $\Adap$. 
    	The second row shows that $\Adap$ outperforms the two unstructured 
    	state-of-the-art methods $\Hp$ and $\Hy$. 
    }
    \label{fig:perf_HB}
\end{figure*}
		
		When comparing the five structured strategies to each other, we find that they have different strengths and weaknesses. 
		For instance, while $\Bz$, $\Bu$ and \ref{alg_adaptivechoiceoftau} are superior in terms of accuracy, 
		their worst-case behavior wrt. run-time is inferior. 
		$\Bp$ is the only method that solves all problems within 3 times the minimal run-time, but at the cost of a lower accuracy, whereas \ref{alg_adaptivechoiceoftau} solves all problems within 5 times the minimal run-time (not depicted) but at high accuracy.
		
		\paragraph{Comparison with existing structured \LBFGS~methods}\label{sec_compwithPetramethod}
		We compare a matrix free form of the \PetraM~and~\PetraP~methods from \cite{BDLP22} to Algorithm~\ref{alg_hybrid}. 
		To this end, we derive four choices of $\tau_k$ for \PetraM~and~\PetraP~based on the secant equations that are used in the derivation of these methods; cf. \cref{sec_difftoPetra}. In particular, we recover two of the choices made in \cite{BDLP22}.
		The authors of \cite{BDLP22} propose to use a scaled identity as seed matrix in \PetraM~to avoid solving linear systems in the compact representation. For a fair comparison with Algorithm~\ref{alg_hybrid} we use \PetraM~with seed matrix $B_k^{(0)}=\tau_k I+\nabla^2\CS(x_k)$.  
		For \PetraP~we do not check if $B_k$ is positive definite (which is demanded in that method), but only if the search direction is descent, which can be done matrix free. 
		
		Since all methods are identical for quadratic regularizers, we can only use five problems out of 22 to compare the methods. 
		Combining the results from the four choices of $\tau_k$ we obtain \cref{fig:perf_MP}. 
		It shows that our approach requires less run-time while achieving better accuracy.
		
		\pgfplotsset{myaxisstyle/.style={
		axis y line=left,
		scale = 0.55,
		title style={at={(0.8,1.7)}},
		axis lines = left,
		xlabel = $\tau$,
		ymin=0, ymax=1,
		ylabel style={at={(-0.3,1)}},
		xlabel style={at={(1,-0.3)}},
		minor tick num=2,
		grid=both,
		grid style={line width=.1pt, draw=gray!05},
		cycle list name=color list,
		line width=0.8pt,
		x tick label style={
			/pgf/number format/.cd,
			fixed,
			precision=3,
			/tikz/.cd
		},
}}

\begin{figure*}[ht!]
    \centering
    \begin{minipage}[b]{0.32\linewidth}
    \pgfplotstableread[col sep=comma]{data/MP_time.csv}\loadedtable
    \begin{tikzpicture}[font=\footnotesize,inner sep=2pt, outer sep=0pt]
      \begin{axis}[
        myaxisstyle,  
        title = {(a) run-time},
        ylabel = $\rho_s(\tau)$,
        xmin=1, xmax=6, 
      ]
      \addplot[blue] table[x=x, y=c1]{\loadedtable};
      \addplot[red] table[x=x, y=c2]{\loadedtable};
      \addplot[orange] table[x=x, y=c3]{\loadedtable};
      \end{axis}
    \end{tikzpicture}
    \end{minipage}
    \begin{minipage}[b]{0.32\linewidth}
    	\pgfplotstableread[col sep=comma]{data/MP_mTRE.csv}\loadedtable
    	\begin{tikzpicture}[font=\footnotesize,inner sep=2pt, outer sep=0pt]
    	\begin{axis}[
    	myaxisstyle,  
    	title = {(b) accuracy},
    	xmin=1, xmax=1.05,
    	]
      \addplot[blue] table[x=x, y=c1]{\loadedtable};
		\addplot[red] table[x=x, y=c2]{\loadedtable};
		\addplot[orange] table[x=x, y=c3]{\loadedtable};
    	\end{axis}
    	\end{tikzpicture}
    \end{minipage}
	\begin{minipage}[b]{0.32\linewidth}
		\pgfplotstableread[col sep=comma]{data/MP_iter.csv}\loadedtable
		\begin{tikzpicture}[font=\footnotesize,inner sep=2pt, outer sep=0pt]
		\begin{axis}[
		myaxisstyle,  
		title = {(c) iterations},
		xmin=1, xmax=2,
		]
		\addplot[blue] table[x=x, y=c1]{\loadedtable};
		\addplot[red] table[x=x, y=c2]{\loadedtable};
		\addplot[orange] table[x=x, y=c3]{\loadedtable};
		\end{axis}
		\end{tikzpicture}
	\end{minipage}
     \begin{minipage}[b]{\linewidth}
     	\centering
    	\begin{tikzpicture}[font=\footnotesize,inner sep=2pt, outer sep=0pt]
    	\draw[blue,solid] (1,0) -- (1.5,0) node[right] {Proposed} ;
    	\draw[red,solid] (4,0) -- (4.5,0)   node[right] {\PetraM} ;
    	\draw[orange,solid] (7,0) -- (7.5,0)  node[right] {\PetraP} ;
    	\end{tikzpicture}
    \end{minipage}
    \caption{Performance profiles of three structured~\LBFGS~methods, based on the image registration problems with non-quadratic regularizers. 
    	The approach proposed in this paper produces the lowest run-time while yielding the highest accuracy.}
    \label{fig:perf_MP}
\end{figure*}
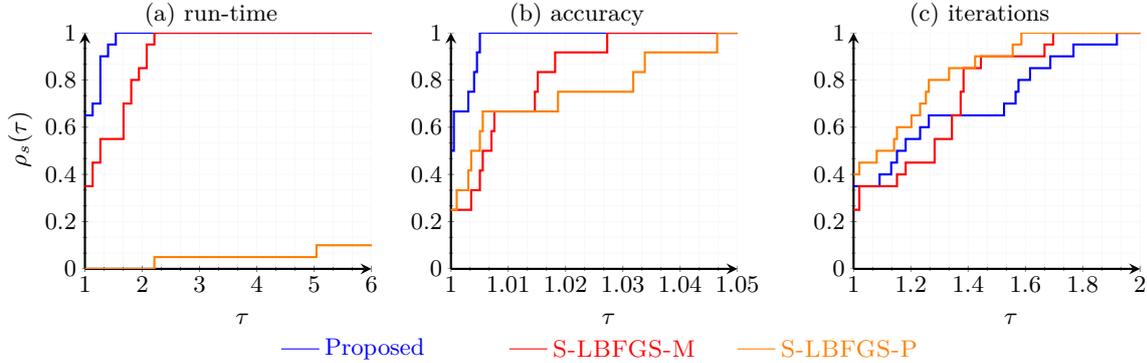
	
	\subsection{Quadratic problems}\label{sec:quad}
		
		We turn to the academic setting of strictly convex quadratics.
		We study how well the Newton direction is approximated, 
		how the regularization parameter and the number of \LBFGS~update vectors affect convergence,
		and if the methods converge linearly as predicted by theory. 
			
		\paragraph{Problems under consideration, setup of experiments}
		We seek the unique minimizer $\xopt:=(1,1,\ldots,1)^T\in\R^{16}$ of the strictly convex quadratic 
		$\CJ(x):=0.5(x-\xopt)^T \left[D+\alpha S\right](x-\xopt)$, where
		$D$ and $S$ are SPD and $\alpha>0$. 
		We let $D$ be the diagonal matrix $D_{jj}:=\exp(-j)$ with exponentially decaying eigenvalues. 
		This matrix is ill-conditioned with a condition number around $10^9$, 
		reflecting the Hessian of a typical data-fidelity term in inverse problems.
		For $S$ we use the classical five-point stencil finite difference discretization of the Laplacian with zero boundary conditions on the unit square \cite[Section~1.4.3]{Ba16}. 
		We investigate three setups: weakly ($\alpha=10^{-5}$), mildly ($\alpha=10^{-3}$) and strongly ($\alpha = 10^{-1}$) regularized problems. 
		We use $x_0=0$, $S_k=\alpha S$, and we terminate if $\|\nabla \CJ(x_k)\| \leq 10^{-13}$. 
		Linear systems are solved with \textsc{Matlab's} backslash.
		
		\paragraph{Proximity to Newton direction}
			
		We study the proximity of $d_k=-B_k^{-1}\nabla \CJ(x_k)$ to the Newton direction 
		$d_k^N:=-\nabla^2\CJ(x_k)^{-1}\nabla \CJ(x_k)$; the latter is optimal as $x_k+d_k^N 
		=\xopt$.
		Specifically, we look at the cosine 
		$d_k^Td_k^N/(\|d_k\|\|d_k^N\|)$ and the ratio $\|d_k\|/\|d_k^N\|$ that ideally are close to one. 
		The experiments do not show any distinctive pattern for these quantities 
		over the iterations, so we use box plots to illustrate them 
		in \cref{fig:boxplot}.
		Each box displays the variation in angle and scale over all iterations, where the middle line of the box 
		represents the median value, the height of the box the interquartile range, 
		and the top and bottom lines the maximum and minimum values.

\newcommand{\alphaA}{10^{-6}}
\newcommand{\alphaB}{10^{-3}}
\newcommand{\alphaC}{10^{-1}}


\newcommand{\cyclelist}{
	{olive},
	{orange},
	{red},
	{violet},
	{cyan},
	{teal},
	{blue},
}

\pgfplotsset{
	compat=1.15,
	my axis style/.style={
		xlabel = iterations ($k$),
		xscale = 0.45,
		yscale = 0.35,
		axis y line=left, axis x line =bottom,
		cycle list/.expanded={\cyclelist},
		every axis title/.style={above, at={(current axis.north)}, yshift = 0cm},
		every axis y label/.style={at={(current axis.west)},xshift=-1cm, rotate=90},
		every axis x label/.style={at={(current axis.south)},yshift=-0.7cm},
	},
}

\newcommand{\myBoxplot}[4]{
	\begin{tikzpicture}
	\begin{axis}[my axis style,#3,boxplot/draw direction=y,boxplot/box extend=0.7,
	xlabel=,xticklabels={,},xmin=0,xmax=8,xtick={1,...,7}]
	\addplot +[boxplot,mark=none,solid] table[x index=0, y index=#4] {data/#1_1_#2.dat};
	\addplot +[boxplot,mark=none,solid] table[x index=0, y index=#4] {data/#1_2_#2.dat};
	\addplot +[boxplot,mark=none,solid] table[x index=0, y index=#4] {data/#1_3_#2.dat};
	\addplot +[boxplot,mark=none,solid] table[x index=0, y index=#4] {data/#1_4_#2.dat};
	\addplot +[boxplot,mark=none,solid] table[x index=0, y index=#4] {data/#1_5_#2.dat};
	\addplot +[boxplot,mark=none,solid] table[x index=0, y index=#4] {data/#1_6_#2.dat};
	\addplot +[boxplot,mark=none,solid] table[x index=0, y index=#4] {data/#1_7_#2.dat};
	\draw [thin, dashed, draw=black] (axis cs: 0,1) -- (axis cs: 7,1);
	\end{axis}
	\end{tikzpicture}
}

\newcommand{\myboxplotLegend}{
	\centering
	\begin{tikzpicture}
	\draw[gray,solid] (1.5,0) -- (2,0) node[right] {$\Hp$} ;
	\draw[blue] (5,0) -- (5.5,0) node[right] {$\Hy$};
	\draw[magenta] (8.5,0) -- (9,0) node[right] {$\Bp$};
	\draw[cyan] (-1,-1) -- (-0.5,-1) node[right] {$\Bz$};
	\draw[red] (1.5,-1) -- (2,-1) node[right] {$\Bu$} ;
	\draw[green] (5,-1) -- (5.5,-1) node[right] {$\GM$} ;
	\draw[orange] (8.5,-1) -- (9,-1) node[right] {$\Adap$};
	\end{tikzpicture}
}

\begin{figure}[ht]
	
	\centering
	(a) cosine with Newton direction
	\pgfplotsset{ymax=1.2,ymin=-1,yticklabels={-1,0,1},ytick={-1,0,1}}
	\def\expname{quadratic_n16__l5}
	\begin{tabular*}{\textwidth}{c @{\extracolsep{\fill}}  c @{\extracolsep{\fill}} c @{\extracolsep{\fill}} c}
		\myBoxplot{\expname}{1}{title = {(i) $\alpha = 10^{-5}$},ylabel={$\ell = 5$}}{4} & 
		\myBoxplot{\expname}{2}{title = {(ii) $\alpha = 10^{-3}$},yticklabels={,}}{4}  & 
		\myBoxplot{\expname}{3}{title = {(iii) $\alpha = 10^{-1}$},yticklabels={,}}{4}
	\end{tabular*}
	
%
	
	(b) scaling ratio with Newton direction\smallskip
	
	\pgfplotsset{ymode=log,ymin=1e-5,ymax=1e1,yticklabels={$10^{-4}$,1},ytick={1e-4,1}}
	\def\expname{quadratic_n16__l5}
	\begin{tabular*}{\textwidth}{c @{\extracolsep{\fill}} c @{\extracolsep{\fill}} c @{\extracolsep{\fill}} c}
		\myBoxplot{\expname}{1}{ylabel={$\ell = 5$},ymode=log}{5} & 
		\myBoxplot{\expname}{2}{ymode=log,yticklabels={,}}{5}  & 
		\myBoxplot{\expname}{3}{ymode=log,yticklabels={,}}{5}
	\end{tabular*}
	
%

	\begin{tikzpicture}
	\draw[olive,solid] (0.5,0) -- (1.0,0) node[right] {$\Hp$} ;
	\draw[orange] (2.0,0) -- (2.5,0) node[right] {$\Hy$};
	\draw[red] (3.5,0) -- (4.0,0) node[right] {$\Bp$};
	\draw[violet] (5,0) -- (5.5,0) node[right] {$\Bz$};
	\draw[cyan] (6.5,0) -- (7,0) node[right] {$\Bu$} ;
	\draw[teal] (8,0) -- (8.5,0) node[right] {$\GM$} ;
	\draw[blue] (9.5,0) -- (10,0) node[right] {$\Adap$};
	\end{tikzpicture}
	
	\caption{Box plots 
		comparing the search directions of 
		unstructured and structured \LBFGS~to the Newton direction, for different values of the regularization parameter $\alpha$.
	}
	\label{fig:boxplot}
\end{figure}
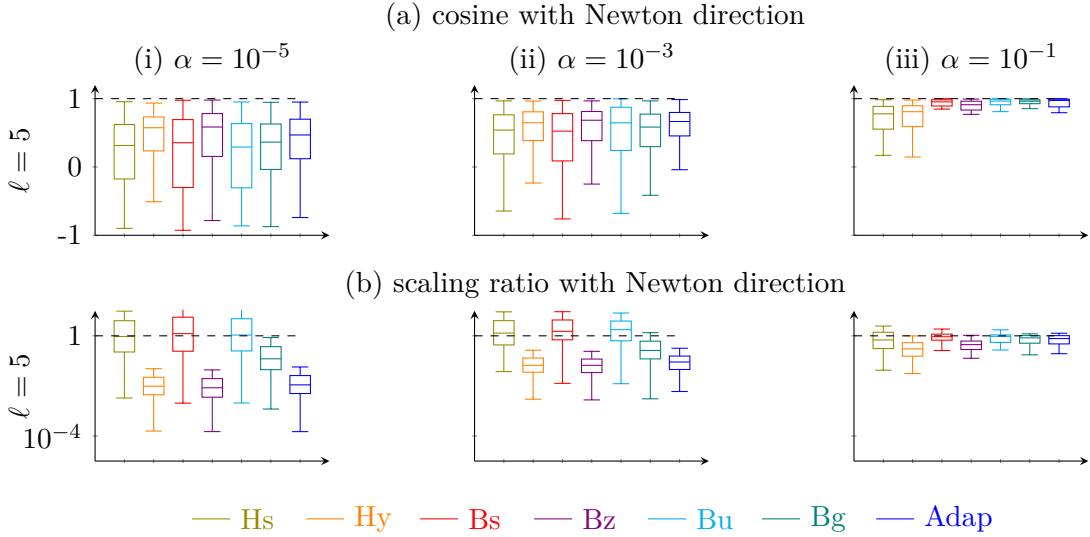

		Increasing $\ell$ yields increasingly better approximations of $d_k^N$ (not depicted). Similar as in the literature for unstructured methods, however, the gains are minor when increasing $\ell$ beyond 5. 
		For the strong regularization of $\alpha=10^{-1}$, the search directions of the structured schemes 
		are closer to the Newton direction than for the unstructured methods. 
		For weak and mild regularization the approximation quality is similar for all schemes.
		All schemes benefit from increasing $\alpha$. 
		
		\paragraph{Effect of regularization strength and memory length}
								
		Iteration numbers, average number of line search steps and run-time are summarized in \cref{tab:quadPrb}. 		
		As expected, the iteration numbers decrease for increasing $\alpha$ and also for for increasing $\ell$. 
		In most cases, the structured schemes require fewer iterations and less run-time than the unstructured ones. 
		Focusing on $\ell=3,5$ with run-time as performance measure, we find that
		$\GM$ is best for $\alpha=10^{-5}$, while $\Bu$ is best for $\alpha=10^{-3},10^{-1}$.
		Moreover, $\GM$ is consistently faster than 
		the unstructured methods if we exclude $\ell=\infty$. For the standard choice $\ell=5$, $\Bp$ is consistently faster than 
		the unstructured methods. 
		
		\begin{table}[h!]
\footnotesize
\centering
\caption{Results on quadratic problems for different values of $\alpha$ and $\ell$. 
By and large, the structured methods outperform the unstructured ones. In particular,
$\GM$ with $\ell<\infty$ and $\Bp$ with $\ell=5$ have lower run-time than the unstructured methods (Hs and Hy) in all cases.
%
}\label{tab:quadPrb}

\setlength{\tabcolsep}{4.5pt}
\newcolumntype{R}{>{\raggedleft\arraybackslash}p{7mm}}
\begin{tabular}{l*{4}{R}c@{\hskip2mm}*{4}{R}c@{\hskip2mm}*{4}{R}}
\hline
\multicolumn{15}{c}{number of iterations}	
\\
\hline
 &\multicolumn{4}{c}{$\alpha=10^{-5}$}
&&\multicolumn{4}{c}{$\alpha=10^{-3}$}
&&\multicolumn{4}{c}{$\alpha=10^{-1}$}
\\
	S\quad/$\ell$ 	&$3$     	&$5$     	&$10$    	&$\infty$  	
	&        	&$3$     	&$5$     	&$10$    	&$\infty$  	
	&        	&$3$     	&$5$     	&$10$    	&$\infty$  
\\
\cline{1-5}
\cline{7-10}
\cline{12-15}
 Hs    &  3380 &  1930 &   846 &    43  &&  478 &   279 &    136 &    34  &&   100 &   87 &    67 &    30  \\
 Hy    &  2950 &  2369 &   1359 &    95  &&  639 &   421 &   211 &    50  &&   107 &  91 &    55 &    33  \\
 Bs    &  2896 &   1762 &   594 &    40  &&  420 &    214 &    85 &    26  &&   33 &    28 &    18 &    15  \\
 Bz    &  2898 &  2560 &  1463 &    93  &&  592 &   439 &   252 &    76  &&   84 &   55 &    46 &    26  \\
 Bu    &  2689 &   2241 &   747 &    42  &&  391 &    172 &    74 &    26  &&   33 &    24 &    18 &    15  \\
 Bg    &  2419 &   1560 &   669 &    63  &&  440 &  248 &    105&    41  &&   41 &   33 &    23 &    18  \\
 Adap  &  2406 &   2211 &  1075 &   66  &&  465 &  350 &    209 &  32  &&   38 &  27 &     20 &   15  \\
\hline
\multicolumn{15}{c}{average number of line searches per iteration}	
\\
\hline
 Hs    &  4.71 & 6.55 & 8.46 & 1.14 && 3.72 & 3.96 & 3.32 & 1.15 && 1.69 & 1.69 & 1.58 & 1.33 \\
 Hy    &  1.10 & 1.10 & 1.13 & 1.00 && 1.11 & 1.13 & 1.18 & 1.00 && 1.07 & 1.04 & 1.07 & 1.03 \\
 Bs    &  4.65 & 6.71 & 8.02 & 1.48 && 3.45 & 3.71 & 3.01 & 1.46 && 1.21 & 1.18 & 1.22 & 1.13 \\
 Bz    &  1.09 & 1.09 & 1.10 & 1.16 && 1.07 & 1.05 & 1.07 & 1.12 && 1.02 & 1.04 & 1.04 & 1.08 \\
 Bu    &  4.73 & 6.86 & 8.21 & 1.40 && 3.39 & 3.60 & 2.89 & 1.46 && 1.27 & 1.21 & 1.22 & 1.13 \\
 Bg    &   2.58 & 3.19 & 3.78 & 1.25 && 1.86 & 1.91 & 1.97 & 1.27 && 1.12 & 1.06 & 1.09 & 1.11 \\
 Adap  &  1.16 & 1.16 & 1.17 & 1.24 && 1.16 & 1.18 & 1.22 & 1.34 && 1.13 & 1.15 & 1.10 & 1.13 \\
\hline 
\multicolumn{15}{c}{run-time (in milliseconds)}	
\\
\hline
Hs    &  879 & 505 & 232 & 15 && 123 & 78 & 40 & 10 && 27 & 24 & 20 & 10 \\
Hy    &  704 & 570 & 329 & 31 && 150 & 103 & 57 & 15 && 29 & 25 & 17 & 11 \\
Bs    &  741 & 466 & 163 & 14 && 109 & 58 & 25 & 9 && 10 & 9 & 6 & 5 \\
Bz    &  669 & 610 & 352 & 29 && 144 & 107 & 67 & 24 && 26 & 20 & 14 & 8 \\
Bu    & 709 & 571 & 201 & 14 && 101 & 47 & 22 & 9 && 10 & 8 & 6 & 5 \\
Bg   & 600 & 387  & 175 & 19 && 109 & 67 & 30 & 13 && 13 & 11 & 7 & 6 \\
Adap & 599 & 535 & 261 & 21 && 117 & 88 & 57 & 12 && 11 & 9 & 7 & 5 \\
\hline
\end{tabular}
\end{table}	
			
		\paragraph{Rate of convergence}		
				
		In \cref{fig_conv} we assess the rate of convergence of the \LBFGS~schemes for $\alpha=10^{-1}$. 
		Unsurprisingly, the structured schemes exhibit better convergence rates than the unstructured ones. 
		$\Bz$ performs worse than the other structured schemes, which are all quite similar to each other. 
		\Cref{thm_rateofconvstrongminimizer} implies q-linear convergence of $(\CJ(x_k))$ 
		and r-linear convergence of $(\norm{x_k-\xopt})$ and $(\norm{\nabla\CJ(x_k)})$, 
		which fits well with \cref{fig_conv}. 
						

\renewcommand{\cyclelist}{
	{olive},
	{orange},
	{red},
	{violet},
	{cyan},
	{teal},
	{blue}
}

\pgfplotsset{
	compat=1.15,
	my axis style/.style={
		xlabel = $k$,
		xscale = 0.5,
		yscale = 0.7,
		axis y line=left, axis x line =bottom,
		cycle list/.expanded={\cyclelist},
		every axis title/.style={above, at={(current axis.north)}, yshift = 0cm},
		every axis y label/.style={at={(current axis.west)},xshift=-1cm, rotate=90},
		every axis x label/.style={at={(current axis.south)},yshift=-0.7cm},
	},
}

\renewcommand{\myBoxplot}[4]{
	\begin{tikzpicture}
	\begin{axis}[my axis style,#3,ymode=log,
	]
	\addplot +[mark=none,solid] table[x index=0, y index=#4] {data/#1_1_#2.dat};
	\addplot +[mark=none,solid] table[x index=0, y index=#4] {data/#1_2_#2.dat};
	\addplot +[mark=none,solid] table[x index=0, y index=#4] {data/#1_3_#2.dat};
	\addplot +[mark=none,solid] table[x index=0, y index=#4] {data/#1_4_#2.dat};
	\addplot +[mark=none,solid] table[x index=0, y index=#4] {data/#1_5_#2.dat};
	\addplot +[mark=none,solid] table[x index=0, y index=#4] {data/#1_6_#2.dat};
	\addplot +[mark=none,solid] table[x index=0, y index=#4] {data/#1_7_#2.dat};
	\end{axis}
	\end{tikzpicture}
}

\renewcommand{\myboxplotLegend}{
	\centering
	\begin{tikzpicture}
	\draw[gray,solid] (1.5,0) -- (2,0) node[right] {$\Hp$} ;
	\draw[blue] (5,0) -- (5.5,0) node[right] {$\Hy$};
	\draw[magenta] (8.5,0) -- (9,0) node[right] {$\Bp$};
	\draw[cyan] (-1,-1) -- (-0.5,-1) node[right] {$\Bz$};
	\draw[red] (1.5,-1) -- (2,-1) node[right] {$\Bu$} ;
	\draw[green] (5,-1) -- (5.5,-1) node[right] {$\GM$} ;
	\draw[orange] (8.5,-1) -- (9,-1) node[right] {$\Adap$};
	\end{tikzpicture}
}

\begin{figure}[ht!]
	
	\centering
	\def\expname{quadratic_n16__l5_conv}
	\begin{tabular*}{\textwidth}{c @{\extracolsep{\fill}}  c @{\extracolsep{\fill}} c @{\extracolsep{\fill}} c}
		\myBoxplot{\expname}{3}{title = {(i) $\|x_k - x^*\|$}}{1} & 
		\myBoxplot{\expname}{3}{title = {(ii) $\CJ(x_k)-\CJ(\xopt)$}}{2}  & 
		\myBoxplot{\expname}{3}{title = {(iii) $\|\nabla \CJ(x_k)\|$}}{3}
	\end{tabular*}

	\begin{tikzpicture}
		\draw[olive,solid] (0.5,0) -- (1.0,0) node[right] {$\Hp$} ;
		\draw[orange] (2.0,0) -- (2.5,0) node[right] {$\Hy$};
		\draw[red] (3.5,0) -- (4.0,0) node[right] {$\Bp$};
		\draw[violet] (5,0) -- (5.5,0) node[right] {$\Bz$};
		\draw[cyan] (6.5,0) -- (7,0) node[right] {$\Bu$} ;
		\draw[teal] (8,0) -- (8.5,0) node[right] {$\GM$} ;
		\draw[blue] (9.5,0) -- (10,0) node[right] {$\Adap$};
	\end{tikzpicture}

	\caption{Convergence profiles on the quadratic problem with $\ell=5$ and $\alpha = 10^{-1}$.
	The proposed structured approach results in higher rates of convergence compared to classical \LBFGS.}
	\label{fig_conv}
\end{figure}
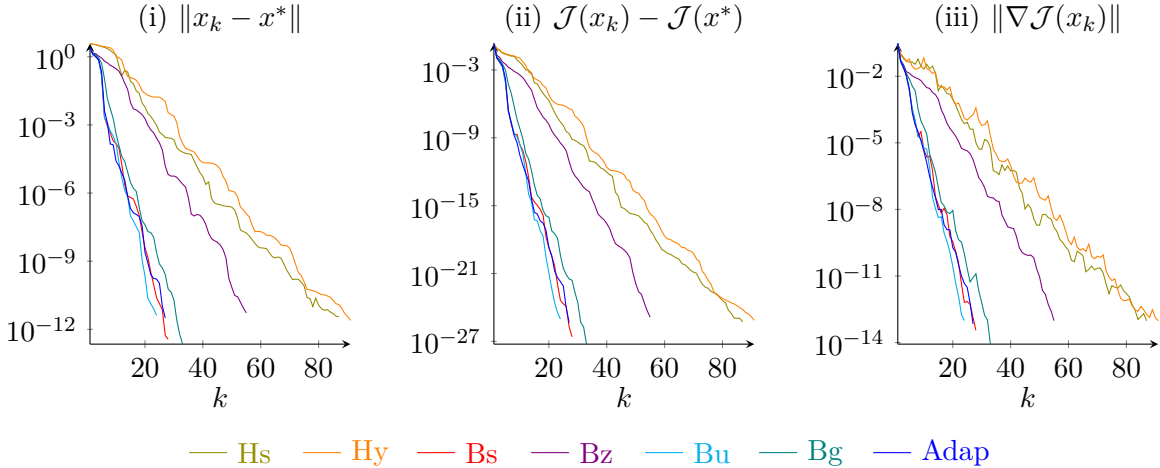
		

\clearpage

\section{Summary and outlook}\label{sec:conclusion}
	
	This paper presents a novel \LBFGS~scheme for structured large-scale optimization problems,
	with the structure under consideration appearing frequently in inverse problems. 
	
	The convergence analysis of this work contributes meaningfully to the literature on \LBFGS. 
	Its main finding is the linear convergence of the proposed method for non-convex problems, which has not been proven
	before for \LBFGS. The assumptions underlying this result are particularly easy to satisfy if the objective includes a convex regularizer
	for which a computationally cheap and uniformly positive definite Hessian approximation is available. 
	Moreover, it is the first convergence analysis for a structured \LBFGS~method, the first convergence analysis for \LBFGS~in general Hilbert spaces, and the first convergence analysis involving the \KL-inequality.
	
	The contributions on the theoretical level are accompanied by extensive numerical experiments.
	Most importantly, we demonstrate that the new method outperforms other structured \LBFGS~methods and classical \LBFGS~on large-scale real-world inverse problems from medical image registration; these problems are highly non-convex and ill-conditioned. 
	The competitiveness to classical \LBFGS~is further illustrated on strongly convex quadratics.
	The method can be implemented matrix free and an implementation is available at \href{https://github.com/hariagr/SLBFGS}{https://github.com/hariagr/SLBFGS}. 	
	
	Future work should study numerically the structured Barzilai--Borwein method that is included in this paper. 
	It could also assess the numerical performance of Algorithm~\ref{alg_hybrid} on additional classes of inverse problems, e.g., problems from PDE-constrained optimal control. From a theoretical perspective it would be interesting to better understand how to tune the scaling parameter $\tau$ of the seed matrix in an effective way (this comment applies to classical \LBFGS~as well) or how to choose the error tolerance of the iterative solver adaptively.
	
%
%

\bibliographystyle{alphaurl}
\bibliography{lit}

\end{document}